\documentclass{amsart}

\usepackage{inputenc}
\usepackage{stmaryrd}
\usepackage{tikz-cd}
\usepackage{tikz}
\usetikzlibrary{cd}
\usepackage{ amssymb }
\usepackage{hyperref}
\hypersetup{
	colorlinks=true,
	filecolor=magenta,      
	urlcolor=cyan,
}
\newtheorem{theorem}{Theorem}[section]
\newtheorem{lemma}[theorem]{Lemma}
\newtheorem{proposition}[theorem]{Proposition}

\theoremstyle{definition}
\newtheorem{definition}[theorem]{Definition}
\newtheorem{cor}[theorem]{Corollary}

\newtheorem{conj}[theorem]{Conjecture}

\theoremstyle{remark}
\newtheorem{remark}[theorem]{Remark}

\newcommand{\spf}{\mathrm{Spf}}
\newcommand{\inj}{\hookrightarrow}

\newcommand{\Of}{\mathcal{O}}
\newcommand{\N}{\mathbb{N}}

\newcommand{\dR}{\mathrm{dR}}
\newcommand{\Q}{\mathbb{Q}}

\newcommand{\Z}{\mathbb{Z}}

\newcommand{\gr}{\mathrm{gr}}

\newcommand{\maxi}{\mathrm{Max}}
\newcommand{\Fil}{\mathrm{Fil}}
\newcommand{\ad}{\mathrm{ad}}

\newcommand{\homo}{\mathrm{Hom}}

\newcommand{\rep}{\mathrm{Rep}}

\newcommand{\wt}{\mathrm{wt}}
\newcommand{\pdr}{\mathrm{pdR}}

\newcommand{\rig}{\mathrm{rig}}
\newcommand{\map}{\rightarrow}
\newcommand{\tri}{\mathrm{tri}}
\newcommand{\ind}{\mathrm{Ind}}

\newcommand{\liealg}{\mathrm{Lie}}

\newcommand{\rob}{\mathcal{R}}

\newcommand{\liet}{\mathfrak{t}}
\newcommand{\rescale}{\mathrm{Res}^K_{\mathbb{Q}_p}}
\newcommand{\phigamma}{(\varphi, \Gamma_K)}
\newcommand{\hphigamma}{\varphi, \gamma_K}
\newcommand{\invertt}{[\frac{1}{t}]}
\newcommand{\rlk}{\rob_{L, K}}
\newcommand{\regu}{\mathrm{reg}}

\title{The Local Companion Points Conjecture}
\author{Lie Qian}

\begin{document}

\begin{abstract}
We describe the set of points of the trianguline variety over a given local Galois representation. Global analogues describing companion points in eigenvariety by \cite{versla} and  \cite{Hansen},  can be thought of as a rational analogue to the weight part of Serre's conjecture. Along the same line, local companion points conjecture can be thought of as a rational analogue of attaching Serre weights to residual Galois representations. \cite{BHS} proves the  conjecture assuming the given Galois representation is cristalline regular. We prove the conjecture in general cases only assuming some regularity conditions.
\end{abstract}

\maketitle

\section{Introduction}
The study of $p$-adic automorphic forms has been central to the global Langlands program. For example, fixing an imaginary CM field $F$ over its totally real subfield $F^+$, one can attach $n$-dimensional global Galois representaions to $p$-adic Hecke-eigenforms of a unitary group $G$ over  $F^+$  that is compact at infinity places and isomorphic to $GL_n$ at $p$-adic places. On the other hand, given an $n$-dimensional global Galois representaions $\rho$ coming from a $p$-adic Hecke-eigenform, one  can naturally ask to give a complete description of all the $p$-adic Hecke-eigenform giving rise to $\rho$. It turns out that although  $\rho$  determines the Hecke eigenvalue away from $p$-adic places, the weight of the $p$-adic Hecke-eigenform might be different, among those giving rise to $\rho$. This question can be thought of as a rational analogue to the weight part of the Serre's conjecture. While the question is hard in nature, we prove in this paper, a local analogue of it, in almost full generality.

Let us be more precise. Fix $L$ be a finite extension of $\Q_p$ with residue field $k$.  Let $\mathcal{E}$ be the eigenvariety asscoiated to $G$ and a prime-to-$p$ level $U^p$. It parametrises $p$-adic Hecke-eigenforms in the continuous function space over $G(F^+)\backslash G(\mathbb{A}_{F^+}^\infty)/U^p$. Each points in $\mathcal{E}(L)$ gives an $n$-dimensional pseudo-representation $\rho$ of $G_F$. Fix a residual representation $\overline{\rho}: G_F\map GL_n(k)$ corresponding to a maximal ideal of the Hecke algebra. There is a component $\mathcal{E}_{\overline{\rho}}$ of $\mathcal{E}$ labelled by $\overline{\rho}$ and it admits a map $p_1: \mathcal{E}_{\overline{\rho}} \map \spf(R_{\overline{\rho}})^{\ad}_\eta$, where $R_{\overline{\rho}}$ is the pseudo-deformation ring of $\overline{\rho}$. Moreover, there is a weight map $p_2: \mathcal{E}_{\overline{\rho}} \map \mathcal{T}^n$, where $\mathcal{T}$ denotes the rigid analytic variety parametrising characters $(F^+\otimes_{\Q}\Q_p)^\times \map  L^\times$. The points of $\mathcal{E}_{\overline{\rho}}$ giving rise to the same $\rho$ are said to be companion points to each other. Thus we are asking for a description of the set $p_2(p_1^{-1}(\rho))$, where $\rho$ is viewed as a point  of $\spf(R_{\overline{\rho}})^{\ad}_\eta$. There is conjectural descriptions of the set \cite[6.5]{versla} (potentially cristalline case) and \cite[Conjecture 1.2.5]{Hansen} (trianguline case). While the precise description is a bit complicated, the point is that the possible weights can be indexed by a subgroup of the Weyl group of $(\mathrm{Res}^{F^+}_{\Q} G)\times_\Q \Q_p$. And the subgroup is in turn determined by $p$-adic Hodge theoretic information of $\rho_{\widetilde{v}}:=\rho_{|_{G_{F_{\widetilde{v}}}}}$ for each $p$-adic places $v$ of $F^+$ with a chosen lift $\widetilde{v}$ of $F$. In other words, the answer is of purely  local (at $p$-adic places) nature.

In \cite{modular}, the authors patchs the eigenvarieties above, and obtain a reduced rigid analytic space $X_p(\overline{\rho})$ that is a union of irreducible components of $\prod_{v|p}X_\tri(\overline{\rho}_{\widetilde{v}})\times \mathbb{U}^g$, where $X_\tri(\overline{\rho}_{\widetilde{v}})$ is the trianguline deformation ring of the local representation $\overline{\rho}_{\widetilde{v}}$ (defined later) and $\mathbb{U}$ is an open unit disc. In fact, it is conjectured that $X_p(\overline{\rho})=\prod_{v|p}X_\tri^\square(\overline{\rho}_{\widetilde{v}})\times \mathbb{U}^g$. This suggests that vaguely, one can view the trianguline deformation ring as a limit of the eigenvarieties, up to a product of open unit disc. Thus one can expect a similar question of companion points for the trianguline variety to be more accessible. Just like the question of finding companion points on eigenvarieties can be thought of as a rational analogue of weight part of Serre's conjecture, their local versions: the question of finding companion points on the trianguline variety, and the question of determining the possible weights of cristalline lifts of a residual local Galois representation (or in the language of Emerton-Gee stack, the Serre weights associated to it \cite[Chapter 8]{EG}), are analogue to each other.

\begin{definition}(\cite[2.2]{modular}, \cite[3.7]{BHS})
Fix a continuous representation $\overline{r}: G_K\map GL_n(k)$, for some finite extension $K/\Q_p$. The trianguline variety $X_\tri(\overline{r})$ is defined as the Zariski closure of $U^\regu_\tri$ in  $\spf(R_{\overline{r}}^\square)^\ad_\eta\times \mathcal{T}^n$. Here, $R_{\overline{r}}^\square$ is the framed local deformation ring of $\overline{r}$, $\mathcal{T}$ is the character variety parametrizing characters $K^\times\map L^\times$, and $U^\regu_\tri$  is defined to be the set of  points $(r, \delta_1, \ldots, \delta_n)$ such that the $\phigamma$-module   $D_{\rig}(r)$ over the Robba ring $\rob_{L, K}$  has a filtration of  $\phigamma$-modules over $\rob_{L, K}$, whose graded pieces are given by rank $1$ $\phigamma$-modules $\rob_{L, K}(\delta_1), \ldots, \rob_{L, K}(\delta_n)$ respectively, and it is also required that all $\delta_i\delta_j^{-1}$ and $\epsilon\delta_i\delta_j^{-1}$ are in $\mathcal{T}_\regu$ (See Definition \ref{regular}, roughly means nonalgebraic), for $i\neq j$. 
\end{definition}

It is the operation of Zariski closure that makes the points in $X_\tri(\overline{r})\backslash U^\regu_{\tri}$ hard to describe. In analogy with the global situation, it is natural to try to describe companion points in the local situation, where we say two points $(r, \delta_1, \ldots, \delta_n)$ and $(r', \delta_1', \ldots, \delta_n')$ of $X_\tri(\overline{r})$ are companion points to each other if $r=r'$. We have two maps $p_1: X_\tri(\overline{r})\map \spf(R_{\overline{r}})^\ad_\eta$ and $p_2: X_\tri(\overline{r})\map \mathcal{T}^n$. Thus we seek a description of $p_2(p_1^{-1}(r))$ for any representation $r$.

To ease the notation, we will only work with the representations $r$ with regular integer Hodge-Tate weights in the introduction. The description of companions points of a representation $r$ will be in terms of data related to Grothendieck-Springer resolution, so we recall some of their construction before stating the main theorem.  Let  $\Sigma$ be the set of embeddings $K\inj L$. Define $\tilde{\mathfrak{g}}_n$ as the closed subscheme of $\mathfrak{gl}_n\times GL_n/B$
given by the points $\{(\psi, gB): \psi\in \mathrm{Ad}(g)(\mathfrak{b})\}$, here we let $\mathfrak{b}$ be the Lie algebra of $B$, the usual upper triangular Borel of $GL_n$, and all the groups and Lie algebra are over $L$.  We let $\tilde{\mathfrak{g}}:=\prod_{\tau\in \Sigma}\tilde{\mathfrak{g}}_n$. $\tilde{\mathfrak{g}}$ can also be viewed as the result of the above construction applied to the group $(\mathrm{Res}^K_{\Q_p}GL_n)_L$. Define $X:=\tilde{\mathfrak{g}}\times_{\prod_{\tau\in \Sigma}\mathfrak{gl}_n}\tilde{\mathfrak{g}}$ and $\mathcal{S}$ to be the absolute Weyl group of $(\mathrm{Res}^K_{\Q_p}GL_n)_L$. $\mathcal{S}$ is isormorphic to $\prod_{\tau\in \Sigma}\mathcal{S}_n$, where $\mathcal{S}_n$ is the usual permutation group acting on $\{1, \ldots, n\}$ and we shall write elements $w\in \mathcal{S}$ as $(w_\tau)_{\tau\in \Sigma}$ under this product. The irreducible components of $X$ are of the same dimension and are labelled by elements $w\in \mathcal{S}$. We write them as $X=\bigcup_{w\in \mathcal{S}}X_w$.

Given any continuous representation $r: G_K\map GL_n(L)$ with integer Hodge-Tate weights. If there exists a filtration $\mathcal{M}_\bullet$ on $\mathcal{M}:=D_\rig(r)\invertt$ by $\rob_{L, K}\invertt$ submodules stable under $\phigamma$-action, whose graded pieces are rank-$1$ $\phigamma$-modules over $\rob_{L, K}\invertt$, then we can define an associated point $x\in X(L)$ as the following: We consider the finite free $K\otimes_{\Q_p}L$-module $D_{\pdr}(r)$ of rank $n$ \cite{Fon}. By the construction of the functor $D_\pdr$, we have Fontaine's operator $N$ acting on $D_{\pdr}(r)$ and there are two filtrations on it: one is given by the usual Hodge filtration, the other is by applying $D_{\pdr}$ (it can actually be defined on the category of  $B_{\dR}$ representation of $G_K$) to the base change of $\mathcal{M}_\bullet$ along $\rob_{L, K}\invertt\map B_{\dR}$. Fontaine's operator preserves both filtration, and thus a choice of trivialization $(K\otimes_{\Q_p}L)^n\cong D_{\pdr}(r)$ pulls Fontaine's operator and the two filtrations back to define a point $x\in X(L)$. In particular, if $D_\rig(r)$ has a filtration $\Fil_\bullet$ of $\phigamma$-submodules over $\rob_{L, K}$, with rank $1$ graded pieces, we can set $\mathcal{M}_\bullet:=\Fil_\bullet\invertt$ and produce a point $x\in X(L)$ as above. This is what will happen in the conjecture below. In this case, we say $x$ is associated with $r$  and we define $\mathcal{S}(x)$ to be the subset of $\mathcal{S}$ consisting of $w$ such that $X_w$ pass through $x$.

The following conjecture is implicit in \cite[4.2]{BHS} and can be thought of as a local analogue of \cite[6.5]{versla} and \cite[Conjecture 1.2.5]{Hansen} . Although we state it here under the condition that the Hodge-Tate weights of $r$ are regular integral, there is a version  that works with arbitrary $r$ with regular non-integral weights, where the meaning of $X$ and the definition of the point $x\in X(L)$ associated with $r$ needs to be modified slightly, see the paragraph preceding Theorem \ref{maintheorem}. However, it is recommended that the reader could assume the weights are integral and skip Section \ref{formdefprob} to simplify the notation.

\begin{conj}\label{conjec}
Given a point $z=(r, \delta_1, \ldots, \delta_n)\in \spf (R^{\square}_{\overline{r}})^{\ad}_\eta(L) \times \mathcal{T}^n(L)$. Assume $r$ is regular  with integral $\tau$-Hodge-Tate-Sen weights $\{h_{\tau, 1}<\cdots <h_{\tau, n}\}$ for each $\tau\in \Sigma$.  The following conditions are equivalent:
\begin{enumerate}
    \item $z\in X_{\tri}(\overline{r})(L)$.
    \item $r$ is trianguline, having a triangulation with parameters $\delta_1', \ldots, \delta_n'$, such that  there exists a permutation $w\in \mathcal{S}(x)$, where $x$ is the associated point in $X(L)$ of $r$, such that for any $i\in \{1, \ldots, n\}$,
    \[
    \delta_i=\delta_i'\prod_{\tau\in \Sigma}x_\tau^{h_{\tau, w_{\tau}(i)}-\wt_\tau(\delta_i')}
    \]
     
\end{enumerate}
\end{conj}

\cite[Theorem 1.7]{BHS} proved this conjecture in the case where $r$ is cristalline regular. Let $w_x$ be the relative position of the two flags given by $x$, a necessary condition for $w\in \mathcal{S}(x)$ is that $w\succ w_x$ \cite[Lemma 2.2.4]{BHS}, in the Bruhat order of $\mathcal{S}$. Note that when $r$ is cristalline, Fontaine's operator is $0$ on $D_\pdr(r)$, and we have  that $w\succ w_x$ is indeed a sufficient condition for $w\in \mathcal{S}(x)$ by simply looking at the Bruhat cells in the slice $\{0\}\times \mathrm{Res}^K_{\Q_p}GL_n/B\times \mathrm{Res}^K_{\Q_p}GL_n/B$. This is how the above conjecture specializes to the formulation in  \cite{BHS} in the cristalline case. We shall also mention the work \cite{Wu} that deals with the case of non-regular cristalline representations. \cite[Corollary 3.7.8]{BHS} also proved $(1)\Rightarrow (2)$, in the case where $r$ has regular integral Hodge-Tate-Sen weights. They constructed a smooth model of the local deformation rings at a point $z\in X_{\tri}(\overline{r})(L)$, in terms of Grothendieck-Springer resolution. We will use this idea heavily in this paper.

The main result of this paper is the following.

\begin{theorem}
Conjecture \ref{conjec} is true. And there is a version for general $r$ which is also true (See Theorem \ref{maintheorem}).
\end{theorem}

Here is  a sketch of the proof. In Section \ref{formdefprob}, we construct deformation spaces of $B_\dR$ representations and all the relevant ones involved in defining the local model, in the broader generality of arbitrary regular weights. From there a similar argument as in \cite[Corollary 3.7.8]{BHS} shows $(1) \Rightarrow (2)$ (See Proposition \ref{nece}). The hard part is to prove $(2)\Rightarrow (1)$, without any cristalline assumption. 

In \cite[Theorem 4.2.3]{BHS}, the authors crucially used a moduli space of refined cristalline representations $\widetilde{\mathfrak{X}}_{\overline{r}}^{\mathbf{h}-\mathrm{cr}}$   which can be described by a space of filtered $\varphi$-module. They constructed  an explicit map from a strata $\widetilde{W}_{\overline{r}, w}^{\mathbf{h}-\mathrm{cr}}$ of it  to $U_\tri^\regu$, associating to a cristalline representation the triangulation induced by a refinements of its $\varphi$-eigenvalues. The map then automatically extend to Zariski closures $\overline{\widetilde{W}_{\overline{r}, w}^{\mathbf{h}-\mathrm{cr}}}\map X_{\tri}(\overline{r})$ and yields explicit points in $X_{\tri}(\overline{r})\backslash U_\tri^\regu$.

We give a completely different approach, working with the trianguline variety itself only, and thus can be made general. Fix a point $z=(r, \delta_1, \ldots, \delta_n)\in \spf (R^{\square}_{\overline{r}})^{\ad}_\eta(L) \times \mathcal{T}^n(L)$. For each $c\in \N$, we construct a moduli problem $M_c$ (Definition \ref{moduliprob}) that roughly classifies points $(r, \delta_1, \ldots, \delta_n)$ such that $D_\rig(r)\invertt$ has a triangulation with parameter $\delta_1, \ldots, \delta_n$ and that the $\tau$-weights of $r$ matches the $\tau$-Hodge-Tate weights $\{\wt_\tau(\delta_1), \ldots, \wt_\tau(\delta_n)\}$. We prove that  $M_c$ is a Zariski locally closed subspace of $\spf (R^{\square}_{\overline{r}})^{\ad}_\eta(L) \times \mathcal{T}^n(L)$ for any $c$. We have $U_\tri^\regu\subset M_c$ for each $c$, and  $M_c$ can be thought of as a first guess to what the closure of $U_\tri^\regu$ are, although they can be much larger than the actual closure $X_\tri(\overline{r})$.  If the condition (2) is satisfied for our fixed point $z$, then it is easy to see $z\in M_c(L)$ for $c$ large enough. We need to narrow it down to $X_\tri(\overline{r})$. 

The crucial idea is to use a dimension argument to produce enough points in $U_\tri^\regu$. This is done by analyzing the local deformation ring given by the completion $\widehat{M}_{c, z}$ at $z$ and hence the local geometry of $M_c$ near $z$: Proposition \ref{descpmc} gives a complete description of $\widehat{M}_{c, z}$ in terms of Grothendieck-Springer resolution. In particular, it has a smooth model given by a closed subset of $\widehat{X}_x$, the completion of $X$ at the point $x$ that is associated with our point $z$. Proposition \ref{existmaxdim} then shows that if we have the $w\in \mathcal{S}(x)$ as in condition (2), the irreducible component of maximal dimension $\widehat{X}_{w, x}$ appears in the smooth model, and $\widehat{M}_{c, z}$ has an irreducible component of maximal dimension $=\dim X_\tri(\overline{r})$. This tells us that $M_c$ has an irreducible components $Y$ of maximal dimension passing through $z$. Thus we shall conclude if we can prove $U_\tri^\regu\cap Y$ is dense in $Y$. By the construction of $M_c$, we also prove that those "junk" points  $M_c\backslash U_\tri^\regu$ are covered by a countable union of Zariski locally closed subspace of dimension $<\dim X_\tri(\overline{r})$ (Proposition \ref{covering} and Lemma \ref{closedembed}). However, if $U_\tri^\regu\cap Y$ is not dense in $Y$, its closure is of smaller dimension and $Y\subset M_c$ can  thus be covered by a countable union of Zariski closed subspace of  dimension $<\dim X_\tri^\square=\dim Y$. This is absurd by Lemma \ref{nocover}. Hence, $Y\subset X_\tri(\overline{r})$, the closure of $U_\tri^\regu$, and in particular $z\in X_\tri(\overline{r})(L)$.

\subsection*{Notations and Conventions}
We fix a $p$-adic local field $K$, and we will only consider Galois representations of $G_K$ in this paper. Also fix a $p$-adic local field $L$, that will be our coefficients field, and we assume it to be large enough.  In particular, we require all embeddings $\tau: K\map \overline{\Q}_p$ have image in $L$. Set $\Sigma:=\mathrm{Hom}(K, L)$ be the set of all embeddings. Let $k$ be the residue field of $L$. Denote by $C$ the completion of $\overline{\Q}_p$. We often use $A$ or $R$ to denote an $L$-Banach algebra. The convention of Hodget-Tate-Sen weights will be that cyclotomic character have Hodge-Tate weight $1$. A continuous representation $r:G_K\map GL_n(L)$ is called regular if for any embedding $\tau: K\map L$, the $\tau$-Hodge-Tate-Sen polynomial has distinct roots in $A$. For a continuous character $\delta: K^\times\map R^\times$, the $\tau$-weight $\wt_\tau(\delta)$ is defined such that $e_\tau$ acts on $R$  by the scalar $\wt_\tau(\delta_i)$, where $e_\tau$ is the idempotent corresponding to the $\tau$-factor in the decomposition $(\mathrm{Lie \Of_K^\times})\otimes_{\Q_p} R\cong K\otimes_{\Q_p}R\cong \bigoplus_{\tau\in \Sigma} R$, which acts on $R$ by differentiating the action of $\Of_K^\times$ via $\delta$. For any $\phigamma$-module $D$ over $X$, cohomology groups are defined as in \cite{KPX}. We denote these by $H^\ast_{\hphigamma}(D)$.

Following \cite{KPX}, a $\phigamma$-module over a Banach algebra $A$ is a finite projective module over $\rob_{A, K}$ with commuting  semilinear linear actions of $\varphi$ and $\Gamma_K$. Here $\rob_{A, K}$ is the relative Robba ring defined as in \cite[Definition 2.2.2]{KPX}. We sometimes also call it a  $\phigamma$-module over $\rob_{A, K}$. For any continuous representation $r: G_K\map GL_n(A)$, there is an associated $\phigamma$-module $D_\rig(r)$ over $A$. Let $t=\log (1+T)\in \rob_{L, K}$. It decomposes as $t=\prod_{\tau\in \Sigma}t_\tau$ in $\rob_{L, K}$, see the paragraph in \cite{Bergdall} before Section 2.5.  We define a $\phigamma$-module over $\rob_{A, K}\invertt$ (resp. $\rob_{A, K}/t$) to be a finite projective module over $\rob_{A, K}\invertt$ (resp. $\rob_{A, K}/t$) with commuting  semilinear linear actions of $\varphi$ and $\Gamma_K$. We say a $\phigamma$-module $M$ over $\rob_{A, K}$ (resp. $\rob_{A, K}\invertt$) is trianguline with parameters $\delta_1, \ldots, \delta_n$ if it admits a filtration $\Fil_\bullet$ of $\phigamma$-modules over $\rob_{A, K}$ (resp. $\rob_{A, K}\invertt$), such that the graded pieces $\gr^i_{\Fil_\bullet}\cong \rob_{A, K}(\delta_i)$ (resp. $\rob_{A, K}(\delta_i)\invertt$) for each $i$ (See \cite[Construction 6.2.4]{KPX} for the definition of the character twist $\rob_{A, K}(\delta_i)$). This filtration is also called a triangulation of $M$ (in both cases). There is a globalized version of  all the notions defined above, which works for a general rigid analytic space $X$. See \cite{KPX} for the definition of $\rob_X^{[a, n]}$ and $\rob_X^r$.

For any rigid analytic space $X$. Given a classical point $z\in X$, we let $\kappa_z$ be the residue field at $z$ and $\widehat{X}_z$ be the completion of $X$ at $z$. Let $\mathcal{C}_L$ be the category of finite dimensional local Artinian $L$-algebra with residue field $L$.  And for any coherent sheaf $C$ on $X$, let $C_z$ be its base change to $\kappa_z$. More generally, for any map $\maxi(S)\map X$ and a coherent sheaf $C$ (resp. map of coherent sheaf $f:  C\map D$) on $X$, let $C_S$ (resp. $f_S: C_S\map D_S$) denote the base change of $C$ (resp. $f$) to $\maxi(S)$.

For any formal scheme $\spf(R)$ over $\Z_p$, we let $(\spf R)^{\ad}_\eta$ be its rigid generic fiber.

Let $D_\pdr$ be Fontaine's almost de Rham functor \cite{Fon}. Let $D_{\mathrm{HT}}$ (resp. $D_{\mathrm{Sen}}$) be the Hodge-Tate functor (resp. Sen's decompletion) on a semilinear representation over $C$. We abuse notation to let $D_{\mathrm{Sen}}(V):=D_{\mathrm{Sen}}(V\otimes C)$ for a local Galois representation $V$. And we set $D_{\mathrm{Sen},n}(M):=M\otimes_{\rob_{L, K}, \theta\circ \varphi^{-n}}(L\otimes_{\Q_p}K(\zeta_{p^n}))$  for a $\phigamma$-module $M$ over $\rob_{L, K}$ and large enough $n$.

Let $\epsilon$ be the cyclotomic character. We say a character $\chi: G_K\map L^\times$ is algebraic if it is given by $x\mapsto \prod_{\tau\in \Sigma}\tau(x)^{a_\tau}$ when viewed as a character $K^\times\map L^\times$ via local class field theory, where $a_\tau\in \Z$ for any $\tau$. Let $\mathcal{T}$ be the rigid analytic space parametrizing continuous characters of $K^\times$ and $\mathcal{T}^n$ the $n$-fold product of it. 

For any ring $R$ and a free module $M$ of rank $d$ with a submodule $M_1$ such that $M/M_1$ is free of rank $i$, there is an associated map $f: \wedge^iM\map R$ and this  map uniquely determines $M_1$. We will say $M_1$ is given by $f$ if $f$ is the associated map as above.

Let $\widetilde{\mathfrak{g}}_n$ be the Grothendieck Springer resolution associated to the group $GL_{n /L}$ defined as in the Introduction.

\subsubsection*{Acknowledgement} I would like to thank Matt Emerton, Eugen Hellmann and Richard Taylor for their interest in this project and valuable discussions. Thanks also to Kevin Lin for a helpful discussion on Grothedieck-Springer resolution. 

\section{Definition of a moduli problem}\label{modu}
\begin{lemma}
Let $X$ be a rigid $L$-analytic space and $\phi: C\map D$ be a homomorphism of locally free coherent sheaves on $X$. Assume that for any classical points $z\in X$, the kernel of the map $\phi_z: C_z\map D_z$ is a $\kappa_z$-vector space of dimension $\leq 1$. Consider the (Zariski sheafification of) moduli problem $\mathcal{P}$ over $X$ that associates to any $f: \mathrm{Max}(S)\map X$ the set of sections $s\in \ker (\phi_S)$ such that $s$ generates  $\ker (\phi_S)$ as a rank-$1$ free module over $S$.

Then $\mathcal{P}$ is represented by a $\mathbb{G}_m$-torsor $\mathcal{L}$ over a Zariski-closed analytic subspace $Y$ of $X$
\end{lemma}

\begin{proof}
We may work locally. So we assume $X=\mathrm{Max}(R)$  for a $L$-Banach algebra $R$,  $C=R^n$, $D=R^m$ and $\phi$ is given by a matrix $\Phi\in M_{m\times n}(R)$. It follows from linear algebra and the dimension assumption on the kernel that at each classical points $z\in X$, the image of the Fitting ideal $\mathrm{Fitt}_1(\Phi)$ (the ideal generated by all $(n-1)\times (n-1)$ minors) is nonzero. Thus $\mathrm{Fitt}_1(\Phi)=(1)$. We can further shrink $X$ and switch basis to assume that $\Phi$ can be written in block form
\[
\left( 
\begin{matrix}
    A& B\\
    C& D
\end{matrix}
\right)
\]
where $A$ is an $(n-1)\times (n-1)$ invertible matrix. One can further left multiply $\Phi$ by the invertible $m\times m$ matrix
\[
\left( 
\begin{matrix}
    A^{-1}& 0\\
    -CA^{-1}& I_{m-n}
\end{matrix}
\right)
\]
and the product is an $m\times n$ matrix $\Phi'$ of the form 
\[
\left( 
\begin{matrix}
    I_{n-1}& B'\\
    0& D'
\end{matrix}
\right)
\]
where $B'$ (resp. $D'$) is a column vector of length $n-1$ (resp. $m-n+1$). Write $D'=[d_1, \ldots, d_{m-n+1} ]^{T}$. Then for any $f: R\map S$, we have $\ker \Phi_S=\ker \Phi'_S=\mathrm{Ann}_S((d_1, \ldots, d_{m-n+1}))$ by taking the last coordinate in $C_S=S^n$. If $s\in \ker \Phi_S$ generates it as a rank-$1$ free module, one has an induced isomorphism of $S$-modules $h: S\map \mathrm{Ann}_S((d_1, \ldots, d_{m-n+1}))$. This implies all $d_i=0$ in $S$ because otherwise for some nonzero $d_i$, $h(d_i)=d_ih(1)=0$. On the other hand, it is clear that all $d_i=0$ implies that $\ker \Phi_S$ is a rank-$1$ free $S$-module. Thus $Y=\mathrm{Max}(R/(d_1, \ldots, d_{m-n+1}))$ and  the moduli problem $\mathcal{P}$ is represented by $Y\times \mathbb{G}_m$ locally.
\end{proof}

\begin{cor}\label{locclosed}
Let $X$ be a rigid $L$-analytic space and $\phi: C\map D$ be a homomorphism of locally free coherent sheaves on $X$. There exists a unique Zariski locally closed analytic subspace $Y$ of $X$ such that for any $f:\maxi(S) \map X$, the $S$-module $\ker(\phi_S)$ is projective of rank $1$ if and only if $f$ factor through $Y$.
\end{cor}

\begin{proof}
We work locally and assume $X=\mathrm{Max}(R)$  for a $L$-Banach algebra $R$,  $C=R^n$, $D=R^m$ and $\phi$ is given by a matrix $\Phi\in M_{m\times n}(R)$.  We see that $\mathrm{Fitt}_1(\Phi_S)=(1)$ is equivalent to $\ker(\Phi_z)$ is of dimension $\leq 1$ for any classical points $z\in \maxi(S)$. Thus $f$ must factor through the Zariski open analytic subspace $U$ given by the complement of the vanishing locus of $\mathrm{Fitt}_1(\Phi)$. Now we argue as in the above lemma, working with $U$ instead of $X$, and find a universal Zariski closed analytic subspace $Y$ of $U$ where $\ker(\phi_S)$ is projective of rank $1$. 
\end{proof}

\begin{remark}
\begin{enumerate}
    \item In the setting of the above corollary, it is straightforward to see that the set of classical points $z\in Y$ is precisely the set of closed points $z\in X$ such that  $\ker \phi_z$ is of dimension $1$ over $\kappa_z$. However, this property does not uniquely determine $Y$ as $Y$ might be nonreduced. For example, one can let $X:=\maxi(L\langle T\rangle)$ and $\phi: \Of_X\map \Of_X$ be  the multiplication of $T^2$. Corollary \ref{locclosed} will produce $Y=\maxi(L[T]/T^2)$.
    \item To ease the notation, when we write "there is a universal Zariski-locally-closed analytic subspace $Y$ of $X$ satisfying property $\mathcal{P}$", we mean "There exists a unique Zariski locally closed analytic subspace $Y$ of $X$ such that for arbitrary (not necessarily locally closed) $f:\maxi(S) \map X$, property $\mathcal{P}$ holds over $\maxi(S)$ if and only if $f$ factor through $Y$".
\end{enumerate}
\end{remark}

\begin{cor}\label{homlocclosed}
Let $X$ be a rigid $L$-analytic space, $M$ be a $(\varphi, \Gamma_K)$-module of rank $d$ over $X$, and $\delta: K^\times \map \mathcal{O}(X)^\times$ be a continuous character. Then there exists a unique maximal Zariski locally closed analytic subspace $Y$ of $X$ such that  $H^0_{\varphi, \gamma_K}(M^\vee (\delta))$ is a rank-$1$ free module over $Y$.
\end{cor}

\begin{proof}
By \cite[Cor 6.3.3]{KPX}, $H^0_{\varphi, \gamma_K}(M^\vee (\delta))$ is locally isomorphic to the kernel of a map between free sheaves and this isomorphism is compatible under pullback. Now we conclude by Corollary \ref{locclosed}.
\end{proof}

\begin{cor}\label{locclosedmodt}
Let $X$ be a rigid $L$-analytic space, $M$ be a $(\varphi, \Gamma_K)$-module of rank $d$ over $X$, and $\delta: K^\times \map \mathcal{O}(X)^\times$ be a continuous character. Then there exists a unique maximal Zariski locally closed analytic subspace $Y$ of $X$ such that  $H^0_{\varphi, \Gamma_K}(M^\vee (\delta))/t_\tau$ is a rank-$1$ free module over $Y$.
\end{cor}

\begin{proof}
Again follows from \cite[Cor 6.3.3]{KPX} and Corollary \ref{locclosed}.
\end{proof}

\begin{proposition}\label{onedfac}
Let $Y$ be a rigid $L$-analytic space, $M$ be a $(\varphi, \Gamma_K)$-module of rank $d$ over $Y$, $\delta: K^\times \map \mathcal{O}(Y)^\times$ be a continuous character such that $H^0_{\varphi, \Gamma_K}(M^\vee (\delta))$ is a line bundle $\mathcal{L}$ over $Y$. Then the canonical map 
\[
M[\frac{1}{t}]\map \mathcal{R}_Y(\delta)[\frac{1}{t}]\otimes_{\mathcal{O}_Y} \mathcal{L}^\vee
\]
is surjective and its kernel is a rank $d-1$ sub-$(\varphi, \Gamma_K)$-module  over $\mathcal{R}_Y[\frac{1}{t}]$ of $M[\frac{1}{t}]$ that is a direct summand.
\end{proposition}

\begin{proof}
The proposition and the proof can be seen as a simplification of \cite[6.3.9]{KPX}, under the stronger condition that $H^0_{\varphi, \Gamma_K}(M^\vee (\delta))$ is a line bundle. Let $Q$ denote the cokernel of the induced map
\[
\lambda: M\map \mathcal{R}_Y(\delta)\otimes_{\mathcal{O}_Y} \mathcal{L}^\vee
\]
. Then it follows from the argument of property (2) in \cite[6.3.9]{KPX} (Page 70) that $Q$ is killed by $t^n$ for some $n\in \N$. Inverting $t$, we see the surjectivity. 

Let $P$ be the kernel of the map $\lambda$. Thus we have an exact sequence 
\[
0 \map P[\frac{1}{t}]\map M[\frac{1}{t}]\map \mathcal{R}_Y(\delta)[\frac{1}{t}]\otimes_{\mathcal{O}_Y} \mathcal{L}^\vee \map 0
\]
It splits as a $\rob_Y\invertt$-module since $\mathcal{R}_Y(\delta)[\frac{1}{t}]\otimes_{\mathcal{O}_Y} \mathcal{L}^\vee$ is a projective 
$\rob_Y\invertt$-module of rank $1$. Thus $P\invertt$ is a rank $d-1$ $\phigamma$-module over $\rob_Y\invertt$ and is a direct summand of $M\invertt$.


\end{proof}




\begin{proposition}\label{modtsurj}
Let $Y$ be a rigid $L$-analytic space, $M$ be a $(\varphi, \Gamma_K)$-module of rank $d$ over $Y$, $\tau \in \Sigma$, $\delta: K^\times \map \mathcal{O}(Y)^\times$ be a continuous character such that $H^0_{\varphi, \Gamma_K}(M^\vee (\delta)/t_\tau)$ is a line bundle $\mathcal{L}$ over $Y$, and that $H^0_{\varphi, \Gamma_K}(M_z^\vee (\delta_z)/t_\tau)$ is $1$-dimensional over $\kappa_z$ for any  $z\in Y$. Then the canonical map  
\[
M/t_\tau\map \rob_Y(\delta)/t_\tau\otimes_{\Of_Y} \mathcal{L}^\vee
\]
is surjective and its kernel is a rank $d-1$ sub-$(\varphi, \Gamma_K)$-module  over $\mathcal{R}_Y/t_\tau$ of $M/t_\tau$ that is a direct summand.
\end{proposition}

\begin{proof}
Denote by $\lambda$ the  canonical map  
\[
M/t_\tau\map \rob_Y(\delta)/t_\tau\otimes_{\Of_Y} \mathcal{L}^\vee
\]
Base change to any points $z\in Y$, we see by Lemma \ref{surjmodt} below that $\lambda_z: M_z/t_\tau\map \rob_{\kappa_z}(\delta_z)/t_\tau$ is surjective: $D_{\mathrm{Sen}, \tau}(\rob_{\kappa_z}(\delta_z))$ is $1$-dimensional over $\kappa_z(\mu_{p^\infty})$, hence the map $\lambda_{z,\mathrm{HT}}:  D_{\mathrm{Sen}, \tau}(M) \map D_{\mathrm{Sen}, \tau}(\rob_{\kappa_z}(\delta_z))$ is surjective if and only if it is non-trivial.

For the surjectivity of $\lambda$, it suffices to work locally, so we may assume $Y$ is an affinoid rigid analytic space. Now $\lambda$ is the base change of some $\lambda^r$ defined over $\rob_Y^r/t_\tau$. Denote the cokernel of $\lambda^r$ by $Q^r$. By the above result on fibers of the map over each $z\in Y$, we see that the base change $Q^{[r/p, r]}$ to $\rob_Y^{[r/p, r]}/t_\tau$ is a coherent sheaf over $\maxi(\rob_Y^{[r/p, r]})$, which is an affinoid rigid analhytic space, and that $Q_z=0$ for any $z\in Y$. Thus $Q$ has empty support over $\maxi(\rob_Y^{[r/p, r]})$ and is thus $0$. By $\varphi$-equivariance, we see immediately that $Q^r$ and hence $Q$ is $0$.

Now $\lambda$ is a surjection from a rank-$d$ projective $\rob_Y/t_\tau$-module  to a rank-$1$ projective $\rob_Y/t_\tau$-module, and the claim about the kernel thus follows.
\end{proof}

\begin{lemma}\label{surjmodt}
Let $f: M\map N$ be a morphism between two   $(\varphi, \Gamma_K)$-module over $\rob_{R, K}/t_\tau$, where $R$ is an affinoid algebra over $L$, and $\tau\in \Sigma$. Then $f$ is surjective (resp. nonzero) if and only if the induced map $f_{\mathrm{HT}}:  D_{\mathrm{Sen}, \tau}(M) \map D_{\mathrm{Sen}, \tau}(N)$ is surjective (resp. nonzero).
\end{lemma}

\begin{proof}
First we reduce to the case $K=\Q_p$: Recall that  for any $(\varphi, \Gamma_K)$-modules $M$ over $\rob_{R, K}/t_\tau$, one define the induced $(\varphi, \Gamma_{\Q_p})$-module $\ind^K_{\Q_p}M:=\ind^{\Gamma_K}_{\Gamma_{\Q_p}} M$ treated as $\rob_{R, \Q_p}/t$-module via the natural inclusion $\rob_{R, \Q_p}/t\map \rob_{R, K}/t_\tau$. It is clear that $f: M\map N$ is surjective if and only if $\ind^K_{\Q_p}f: \ind^K_{\Q_p}M\map \ind^K_{\Q_p}N$ is surjective. On the other hand, since $D_{\mathrm{Sen}}(\ind^K_{\Q_p}M)\cong \oplus_{\tau\in \Gamma_{\Q_p}/\Gamma_K}\tau^*(D_{\mathrm{Sen}, \tau}(M))$ where we view each $\tau^*(D_{\mathrm{Sen, \tau}}(M))$ as a $\Q_p^{\mathrm{cyc}}$ vector spaces via $\Q_p^{\mathrm{cyc}}\map K_\infty$, and $\tau^*(V)$ of a $\Gamma_K$-representation $V$ is defined as $V$ with its $\Gamma_{\tau K}$ action given by precomposing the $\Gamma_K$ action with $\mathrm{Ad}_{\tau^{-1}}: \Gamma_{\tau K}\map \Gamma_K$. It is thus clear that the surjectivity of $f_{\mathrm{HT}}$ is also preserved under induction. Hence we may reduce to the case $K=\Q_p$.

In this case, there is explicit description of $M$, a $(\varphi, \Gamma_{\Q_p})$-module over $\rob_{R, \Q_p}/t$.  By \cite[Lemma 3.2.3]{KPX}, taking colimit over the $r_0$ in the reference, we see that there is a functorial isomorphism 
\[M\cong \mathrm{colim}_m \left(\prod_{n\geq m} D_{\mathrm{Sen}, n}(M)\right)\]
where for $m_2\geq m_1$, the transition map $\prod_{n\geq m_1} D_{\mathrm{Sen}, n}(M)\map \prod_{n\geq m_2} D_{\mathrm{Sen}, n}(M)$ is given by forgetting the terms in the product of index $n < m_2$. One immediately see that in this case, $f: M\map N$ is surjective if and only if the induced $f_{\mathrm{HT}}:  D_{\mathrm{Sen}}(M) \map D_{\mathrm{Sen}}(N)$ is surjective, since the latter is also equivalent to the surjectivity of $D_{\mathrm{Sen}, n}(M)\map D_{\mathrm{Sen}, n}(N)$ for all $n$ large enough.

The claim of the equivalence of nontriviality of the map follows in the same way.
\end{proof}

The following lemma will not be used in this section but will be invoked in Section \ref{laproof}.

\begin{lemma}\label{htwt}
Let $M$ be a $\phigamma$-module over an affinoid algebra $R$ over $L$. And let $\chi$ be a character $K^\times \map R^\times$ whose $\tau$-Hodge-Tate-Sen weight is $\alpha\in R$. Then there exists an surjective morphism  $f: M/t_\tau\map D_{\rig}(\chi)/t_\tau$ preserving  if and only if  $D_{\mathrm{Sen}, \tau}(M)$ has a rank $1$ free $R\otimes_{\tau, K} K_\infty$-module quotient where the Sen operator acts by the scalar $\alpha$.

\end{lemma}

\begin{proof}
By twisting we may assume $\chi$ is trivial. If there exists such an $f$, by looking at the induced map $f_{\mathrm{HT}}$ we see the claim on $D_{\mathrm{Sen}, \tau}$ immediately. On the other hand, for any map $h: D_{\mathrm{Sen}, \tau}(M)\map R\otimes_{\tau, K} K_\infty$ equivariant with respect to the $\Gamma_K$ action, we claim there exists a map $f: M/t_\tau \map \rob_{R, K}/t_\tau$ whose induced map $f_{\mathrm{HT}}=h$. Note that the existence of a surjective $h$ is equivalent to the existence of a map $D_{\mathrm{Sen}, \tau}(M)\map R\otimes_{\tau, K} K_\infty$ that is equivariant with respect to the Sen operator. Granting the claim, the existence of a surjective $f$ follows from Lemma \ref{surjmodt}. 

To prove the claim, the strategy is again to reduce to the case $K=\Q_p$. By taking duals of both $f$ and $h$ and set $N:=M^\vee$, we need to show for any $\Gamma_K$ invariant element  $u\in D_{\mathrm{Sen}, \tau}(N)$, there exist an element $v\in N/t_\tau$ invariant under $\varphi$ and $\Gamma_K$  that specializes to it. The existence of such a $u$ is equivalent to the existence of a $\Gamma_{\Q_p}$-invariant element $u_1\in \ind_{\Gamma_{\Q_p}}^{\Gamma_K}D_{\mathrm{Sen}, \tau}(N)\cong D_{\mathrm{Sen}}(\ind_{\Gamma_{\Q_p}}^{\Gamma_K}N)$, where we view $D_{\mathrm{Sen}, \tau}(N)$ as a $\Q_p^{\mathrm{cyc}}$ vector spaces via $\Q_p^{\mathrm{cyc}}\map K_\infty$. The existence of such a $v$ is equivalent to the existence of a $\varphi$ and $\Gamma_{\Q_p}$-invariant element $v_1\in \ind_{\Gamma_{\Q_p}}^{\Gamma_K}N$. Thus the reduce to the case $K=\Q_p$ (for $(\varphi, \Gamma_{\Q_p})$-modules over $\rob_{R, \Q_p}/t$).  Now by \cite[Lemma 3.2.3]{KPX} again, we have $N\cong \mathrm{colim}_m \left(\prod_{n\geq m} D_{\mathrm{Sen}, n}(N)\right)$. Moreover, the $\varphi$-action is described by the following rule for each $m$:
\[
\prod_{n\geq m} D_{\mathrm{Sen}, n}(N)\map \prod_{n\geq m+1} D_{\mathrm{Sen}, n}(N): (x_n)_{n\geq m}\mapsto (\iota_{n-1}(x_{n-1}))_{n\geq m+1}
\]
where $\iota_j: D_{\mathrm{Sen}, j}(N)\map D_{\mathrm{Sen}, j+1}(N)$ is the natural inclusion map for each $j$. Thus a $\varphi$ and $\Gamma_{\Q_p}$-invariant element $v_1\in N$ must be a constant sequence given by some $\Gamma_{\Q_p}$-invariant $x_m\in D_{\mathrm{Sen}, m}(N)$ in $\prod_{n\geq m} D_{\mathrm{Sen}, n}(N)$, for some $m$. By varying $m$, this is in turn equivalent to a $\Gamma_{\Q_p}$-invariant element  $u_1\in D_{\mathrm{Sen}}(N)$.
\end{proof}

\begin{proposition}\label{existfil}
Let $X$ be a rigid $L$-analytic space, $M$ be a $(\varphi, \Gamma_K)$-module of rank $d$ over $X$, and $\delta_1, \ldots, \delta_d: K^\times \map \mathcal{O}(X)^\times$ be $d$ continuous characters. Then there exists a unique maximal Zariski locally closed analytic subspace $Y$ of $X$ such that the following two conditions are satisfied
\begin{enumerate}
    \item For any $i\in \{1, \ldots, d\}$
\[
\mathcal{L}_i :=\homo_{\varphi, \gamma_K}\left(\bigwedge^{d-i+1}M,\rob_{Y} (\prod_{j=i}^d \delta_j)\right)
\]
is a rank-$1$ locally free module over $Y$
\item There exists an increasing filtration $\mathcal{F}_i$ of $M[\frac{1}{t}]$, for $i\in \{0, \ldots, d\}$, such that $\mathcal{F}_i$ is a $(\varphi, \Gamma_K)$-modules of rank $i$ over $\rob_Y\invertt$ and is a local direct summand of $\mathcal{F}_{i+1}$. We require the filtration $\mathcal{F}_\bullet$ to be compatible with the line bundles $\mathcal{L}_\bullet$, in the sense that there exist local generators $f_i$ of $\mathcal{L}_i$, such that locally each $\mathcal{F}_{i-1}$ is given by $f_i[\frac{1}{t}]$ (See Notation) as a rank $i-1$ sub-$\rob_Y[\frac{1}{t}]$-module of $M\invertt$, for any $i\in \{1, \ldots, d\}$.
\end{enumerate}
\end{proposition}

\begin{proof}
By Corollary \ref{locclosed}, there exists a universal Zariski locally closed analytic subspace $Z$ satisfying (1). In fact, Corollary \ref{locclosed} gives a universal $Y_i$ where  $\mathcal{L}_i$ is locally free of rank $1$, for any $i\in \{1, \ldots, d\}$. One simply take $Z:= \cap_{i=1}^d Y_i$. Here for any finite collection of (not necessarily reduced) Zariski closed sub analytic spaces $\{\maxi(R/I_i)\}_{i=1}^d$, we take their intersection to be $\maxi(R/\left(I_1+\ldots + I_d\right))$.

Now we need to show there is a universal Zariski-Closed analytic subspace of $Z$ satisfying (2). For this we use induction on $i$ starting from $i=d$ to show that there exists a universal Zariski-Closed analytic subspace $Z_i$ of $Z$ where the  rank $j-1$ sub-$\rob_Y[\frac{1}{t}]$-module $\mathcal{F}_{j-1}$ of $M$ given by $f_j[\frac{1}{t}]$ is contained in $\mathcal{F}_{j}$ (defined by  $f_{j+1}[\frac{1}{t}]$) as a local direct summand, for any $j\geq i$. The case $i=d$ is obvious from Proposition \ref{onedfac}. Here different choices of $f_j$ give the same result as we already have the mapping spaces $\mathcal{L}_i$ are rank-$1$ locally free over $Z$. Assuming the claim for $i+1$. By induction hypothesis, we already have $\mathcal{F}_i\subset\cdots\subset \mathcal{F}_d$, where $\mathcal{F}_{j}/\mathcal{F}_{j-1}\cong \rob_{Y} ( \delta_j)\invertt $ for $j=i+1, \ldots, d$. We need to construct $Z_i$ as a universal Zariski-closed analytic subspace of $Z_{i+1}$ where $\mathcal{F}_{i-1}$ given by $f_{i}\invertt$ is a local direct summand contained in $\mathcal{F}_i$.  We recursively construct Zariski-closed $Z_{i, j}$ for $j=d, \ldots, i+1$ and set $Z_i:=Z_{i, i+1}$ in the end. The first step $Z_{i, d}$ is constructed as the zero locus of the map of locally free $\rob_{Z_{i+1}}\invertt$ modules 
\[
f_i\invertt_{|\mathcal{F}_{d-1}}: \bigwedge^{d-i+1}\mathcal{F}_{d-1} \map \rob_{Y} (\prod_{j=i}^d \delta_j)\invertt
\]
Note that $Z_{i, d}$ is Zariski-closed by Lemma \ref{zerolocusclosed}. Now since $\mathcal{F}_{d-1}$ is a rank $d-1$ local direct summand of $\mathcal{F}_d$, we have the exact sequence 
\[
0 \map \bigwedge^{d-i+1}\mathcal{F}_{d-1} \map \bigwedge^{d-i+1}\mathcal{F}_{d} \map \left(\bigwedge^{d-i}\mathcal{F}_{d-1}\right)\otimes (\mathcal{F}_{d}/\mathcal{F}_{d-1}) \map 0
\]
Thus over $Z_{i, d}$, $f_i\invertt$ factors through the quotient fo the above short exact sequence and induces a map 
\[
g_{i, d-1}: \bigwedge^{d-i}\mathcal{F}_{d-1} \map \rob_{Y} (\prod_{k=i}^{d-1} \delta_k)\invertt
\]
. Inductively, we define $Z_{i, j}$ as the zero locus of the map (provided by the induction step for $j+1$) of locally free $\rob_{Z_{i, j+1}}\invertt$ modules 
\[
(g_{i, j})_{| \mathcal{F}_{j-1}}: \bigwedge^{j-i+1}\mathcal{F}_{j-1} \map \rob_{Y} (\prod_{k=i}^{j} \delta_k)\invertt
\]
. Again, one have the short exact sequence 
\[
0 \map \bigwedge^{j-i+1}\mathcal{F}_{j-1} \map \bigwedge^{j-i+1}\mathcal{F}_{j} \map \left(\bigwedge^{j-i}\mathcal{F}_{j-1}\right)\otimes (\mathcal{F}_{j}/\mathcal{F}_{j-1}) \map 0
\]
, which induces a map
\[
g_{i, j-1}: \bigwedge^{j-i}\mathcal{F}_{j-1} \map \rob_{Y} (\prod_{k=i}^{j-1} \delta_k)\invertt
\]
The last step induces a map 
\[
g_{i, i}: \mathcal{F}_{i} \map \rob_{Y} ( \delta_i)\invertt
\]
, realizing $\mathcal{F}_{i-1}$ as a local direct summand of $\mathcal{F}_{i}$. It is straightforward to check that $Z_{i, i+1}$ is the universal Zariski-closed sub-analyitc-space where the  $\mathcal{F}_{i-1}$ satisfy the requirement.
\end{proof}

\begin{lemma}\label{zerolocusclosed}
Let $M, N$ be two $\phigamma$-module over $X$, and $f: M\invertt \map N\invertt$ (resp. $f: M/t_\tau\map N/t_\tau$) a morphism of $\phigamma$-module over $\rob_X\invertt$ (resp. $\rob_X/t_\tau$).   Then there exists a universal Zariski-closed analytic subspace $Y$ in $X$ where the map $f$ is $0$.
\end{lemma}

\begin{proof}
In the case $f: M\invertt \map N\invertt$, since $M\invertt$ is finite projective over $\rob_X\invertt$, there exists some $i$ such that $f$ comes from a map $g: M\map t^{-i}N$ after inverting $t$. Moreover, for any map $Y\map X$, the base change $f_Y: M\otimes_{\rob_X}\rob_Y\invertt\map N\otimes_{\rob_X}\rob_Y\invertt$ is $0$ if and only if $g_Y: M\otimes_{\rob_X}\rob_Y\map t^{-i} N\otimes_{\rob_X}\rob_Y$ is $0$. Now by \cite[Corollary 6.3.3]{KPX}, $\homo_{\varphi, \gamma_K}(M, t^{-i}N)\cong H^0_{\varphi, \gamma_K}(M^\vee\otimes(t^{-i}N))$ is locally isomorphic to the kernel of a map between two free sheaves $\Of_X^m\map \Of_X^n$, which is compatible with base change. Now locally, the desired universal analytic space where $f$ vanishes is clearly the Zariski closed analytic subspace of $X$ defined by the ideal generated by the coordinates of $f$, viewed as an element of $\Of_X^m$.

The argument in the case of $\phigamma$-module over $\rob_X/t_\tau$ follows in the same way, using that $H^0_{\varphi, \gamma_K}(M^\vee/t_\tau\otimes N/t_\tau)$ is locally isomorphic to the kernel of a map between two free sheaves.
\end{proof}

\begin{remark}
In fact, for any map $f: M\invertt \map N\invertt$ (resp. $f: M/t_\tau\map N/t_\tau$) of $\rob_X\invertt$(resp. $\rob_X/t_\tau$)-module, not necessarily preserving $\varphi, \Gamma_K$-action, the zero locus should be Zariski-closed as well. We do not need this stronger fact, so we only prove Lemma \ref{zerolocusclosed}, which has a simpler proof.
\end{remark}

\begin{proposition}
Let $X$ be a rigid $L$-analytic space, $M$ be a $(\varphi, \Gamma_K)$-module of rank $d$ over $X$, $\tau\in \Sigma$ and $\delta_1, \ldots, \delta_d: K^\times \map \mathcal{O}(X)^\times$ be $d$ continuous characters. Then there exists a unique maximal Zariski locally closed analytic subspace $Y$ of $X$ such that the following two conditions are satisfied

\begin{enumerate}
    \item For any $i\in \{1, \ldots, d\}$
\[
\mathcal{L}_i^\tau :=\homo_{\varphi, \gamma_K}\left(\left(\bigwedge^{d-i+1}M\right)/t_\tau,\rob_{Y} (\prod_{j=i}^d \delta_j)/t_\tau\right)
\]
is a rank-$1$ locally free module over $Y$
\item There exists an increasing filtration $\mathcal{F}_i^\tau$ of $M/t_\tau$, for $i\in \{0, \ldots, d\}$, such that $\mathcal{F}_i^\tau$ is a $(\varphi, \Gamma_K)$-modules of rank $i$ over $\rob_Y/t_\tau$ and is a local direct summand of $\mathcal{F}_{i+1}^\tau$. We require the filtration $\mathcal{F}_\bullet^\tau$ to be compatible with the line bundles $\mathcal{L}^\tau_\bullet$, in the sense that there exist local generators $f_i^\tau$ of $\mathcal{L}_i^\tau$, such that locally each $\mathcal{F}_{i-1}^\tau$ is given by $f_i^\tau$  (See Notation) as a rank $i-1$ sub-$\rob_Y/t_\tau$-module of $M/t_\tau$, for any $i\in \{1, \ldots, d\}$.
\end{enumerate}

\end{proposition}

\begin{proof}
The same argument as Proposition \ref{existfil}: replacing the applications of Proposition \ref{onedfac} by the application of Proposition \ref{modtsurj}.
\end{proof}

The previous proposition guarantees that the following definition makes sense.

\begin{definition}\label{moduliprob}
Fix a non-negative integer $c$. We let $M_{c}$ be the unique universal Zariski-locally-closed analytic subspace $\spf (R^{\square}_{\overline{r}})^{\ad}_\eta \times \mathcal{T}^n$  such that the following conditions are satifsied (Note that there are universal $\phigamma$ module $D_\rig(r)$ and characters $\delta_1, \ldots, \delta_n$ over $\spf (R^\square_{\overline{r}})^{\ad}_\eta \times \mathcal{T}^n$ coming from each factor) for any affinoid subdomain $Y\subset \spf (R^\square_{\overline{r}})^{\ad}_\eta \times \mathcal{T}^n$:
\begin{enumerate}
    \item For any $i\in \{1, \ldots, n\}$
\[
\mathcal{L}_i :=\homo_{\varphi, \gamma_K}\left(\bigwedge^{n-i+1}D_\rig(r), t^{-c}\rob_{Y} (\prod_{j=i}^n \delta_j)\right)
\] 
is a rank-$1$ locally free module over $Y$.
\item For any $i\in \{1, \ldots, n\}$ and $\tau\in \Sigma$,
\[
\mathcal{L}_i^\tau :=\homo_{\varphi, \gamma_K}\left(\left(\bigwedge^{n-i+1}D_\rig(r)\right)/t_\tau,\rob_{Y} (\prod_{j=i}^n \delta_j)/t_\tau\right)
\]
is a rank-$1$ locally free module over $Y$.
\item There exists an increasing filtration $\mathcal{F}_i$ of $D_\rig(r)[\frac{1}{t}]$, for $i\in \{0, \ldots, n\}$, such that $\mathcal{F}_i$ is a $(\varphi, \Gamma_K)$-modules of rank $i$ over $\rob_Y\invertt$ and is a local direct summand of $\mathcal{F}_{i+1}$. We require the filtration $\mathcal{F}_\bullet$ to be compatible with the line bundles $\mathcal{L}_\bullet$, in the sense that there exist local generators $f_i$ of $\mathcal{L}_i$, such that locally each $\mathcal{F}_{i-1}$ is given by $f_i[\frac{1}{t}]$ as a rank $i-1$ sub-$\rob_Y[\frac{1}{t}]$-module of $D_\rig(r)\invertt$, for any $i\in \{1, \ldots, n\}$.
\item For any $\tau\in \Sigma$, there exists an increasing filtration $\mathcal{F}_i^\tau$ of $D_\rig(r)/t_\tau$, for $i\in \{0, \ldots, n\}$, such that $\mathcal{F}_i^\tau$ is a $(\varphi, \Gamma_K)$-modules of rank $i$ over $\rob_Y/t_\tau$ and is a local direct summand of $\mathcal{F}_{i+1}^\tau$. We require the filtration $\mathcal{F}_\bullet^\tau$ to be compatible with the line bundles $\mathcal{L}^\tau_\bullet$, in the sense that there exist local generators $f_i^\tau$ of $\mathcal{L}_i^\tau$, such that locally each $\mathcal{F}_{i-1}^\tau$ is given by $f_i^\tau$ as a rank $i-1$ sub-$\rob_Y/t_\tau$-module of $D_\rig(r)/t_\tau$, for any $i\in \{1, \ldots, n\}$.
\end{enumerate}

\end{definition}

\begin{proposition}
For any $x=(r, \delta_1, \ldots, \delta_n)\in M_c(L)$, let $\widehat{M}_{c, x}$ be the completed local ring of $M_c$ at $x$. The moduli problem it represents is the following: For any $A\in \mathcal{C}_L$, $\widehat{M}_{c, x}(A)$ is the isomorphism classes of a continuous representation $r_A: G_K \map GL_n(A)$ and $n$-tuples of characters $\delta_{A, 1 }, \ldots, \delta_{A, n}: K^\times \map A^\times$ lifting $r$ and $\delta_1, \ldots, \delta_n$ respectively, such that all conditions in Definition \ref{moduliprob} are satisfied with $Y:=\maxi(A)$, $r:=r_A$, $\delta_i:= \delta_{A, i}$, for any $i=1, \ldots, n$ and $\tau\in \Sigma$.
\end{proposition}

\begin{proof}
Evident from definition.    
\end{proof}

\begin{remark}
By abuse of notation, we will also use $\widehat{M}_{c, x}$ for the moduli problem it represents.
\end{remark}

\section{The properties of Various Formal Deformation Problems}\label{formdefprob}
In this section, we define several deformation problems similar to those studied in \cite[Section 3]{BHS} , that work in greater generality where the weights are not assumed to be integers.

Let $\Sigma:=\homo_{\Q_p}(K, L)$, the set of embeddings from $K$ to $L$.

\subsection{Twisted Almost de Rham Representations}

 Fix a $\tau\in \Sigma$  in this subsection to simplify notation. In later subsections we will put a $\tau$ as a subscript on the notions defined in this subsection to specify $\tau$. We first define an equivalence relation on the set of characters $\{\delta_1, \ldots, \delta_n\}$.

\begin{definition}\label{intwteq}
\begin{enumerate}
    \item We say two characters $\delta_i$ and $\delta_j: K^\times\map R^\times$, where $R$ is a $L$-Banach algebra, are of integral $\tau$-weights  difference to each other if and only if $\wt_\tau(\delta_i)-\wt_\tau(\delta_j)\in \Z$. This only depend on the isomorphism class of the $\phigamma$-module $\rob_{L, K}(\delta_i)\invertt$ and $\rob_{L, K}(\delta_j)\invertt$ as another choice $\delta'_i$ of the character $\delta_i$ satisfy $\wt_\tau(\delta'_i)-\wt_\tau(\delta_i)\in \Z$ for any $\tau\in \Sigma$.
    \item The set $\{1, \ldots, n\}$ decomposes, according to this equivalence relation, to a disjoint union of sets $S_1, \ldots, S_k$, such that two indices $i$ and $j$ are in a same set $S_l$ for some $l\in \{1, \ldots, k\}$ if and only if $\delta_i$ and $\delta_j$ are of integral $\tau$-weights  difference. Let $n_l$ denote the cardinality of the set $S_l$. So $\sum_{l=1}^k n_l=n$. 
    \item Within the class $r+\Z$ for some $r\in L$, we define a well-ordering $<$ on it induced by the one on $\Z$. 
\end{enumerate}
\end{definition}


We have the isomorphism 
\[
\mathrm{Lie} (\rescale\mathbb{G}_m)^{n} \otimes_{\Q_p}L\cong \bigoplus_{\tau\in \Sigma} \left(\liealg \mathbb{G}_{m, K}\otimes_{K, \tau} L\right)^n\cong \bigoplus_{\tau\in \Sigma}L^n
\]
For our fixed $\tau$, we decompose the factor labelled by $\tau$ in the above isomorphism into 
\[
\left(\liealg \mathbb{G}_{m, K}\otimes_{K, \tau} L\right)^n\cong \left(\liealg \mathbb{G}_{m, K}\otimes_{K, \tau} L\right)^{S_1}\oplus \cdots \oplus \left(\liealg \mathbb{G}_{m, K}\otimes_{K, \tau} L\right)^{S_k}
\]
and we set $\liet_\tau:=\left(\liealg \mathbb{G}_{m, K}\otimes_{K, \tau} L\right)^n$ and $\liet_{\tau, l}:=\left(\liealg \mathbb{G}_{m, K}\otimes_{K, \tau} L\right)^{S_l}$ to be the $l$-th factor for any $l\in \{1, \ldots, k\}$. $\liet_{\tau, l}$ is isomorphic to $L^{n_l}$. We let $\widehat{\mathfrak{t}}_\tau$ (resp. $\widehat{\liet}_{\tau, l})$ be the completion of $\liet_\tau$ (resp. $\liet_{\tau, l}$) at $0$. Set $\widehat{\mathfrak{t}}:=\bigoplus_{\tau\in \Sigma}\widehat{\mathfrak{t}}_\tau$.


We let $\mathcal{S}_{\tau, l}$ be  the permutation group of $S_l$. Also set $\mathcal{S}_\tau :=\prod_{l=1}^k \mathcal{S}_{\tau, l}$ and $\mathcal{S}:=\prod_{\tau\in \Sigma}\mathcal{S}_\tau$. $\mathcal{S}_{\tau, l}=\mathrm{Aut}(S_l)$ naturally acts on $\liet_{\tau, l}$ by permuting coordinates. Let $T_{\tau, l}:=\mathfrak{t}_{\tau, l}\times_{\mathfrak{t}_{\tau, l}/\mathcal{S}_{\tau, l}}\mathfrak{t}_{\tau, l}$, and for a $w_{\tau, l}\in \mathcal{S}_{\tau, l}$, we let $T_{w_{\tau, l}}=\{(z, \mathrm{Ad}(w_l^{-1})z), z\in \mathfrak{t}_{\tau, l}\}$ be the irreducible component of $T_{\tau, l}$ indexed by $w_{\tau, l}$. Also let   $\widehat{T}_{\tau, l,(0, 0)}$ (resp. 
 $\widehat{T}_{w_{\tau, l}, (0, 0)}$) be the completion of $T_{\tau, l}$ (resp. $T_{w_{\tau, l}}$) at the point $(0, 0)$. We set  $T_\tau :=\prod_{l=1}^k T_{\tau, l}$ (resp. $T=\prod_{\tau\in \Sigma}T_\tau$) and for each $w_\tau=(w_{\tau, l})\in \mathcal{S}_\tau$ (resp. $w=(w_\tau)_\tau\in \mathcal{S}$), we set $T_{w_\tau}:=\prod_{l=1}^k T_{w_{\tau, l}}$ (resp. $T_w:=\prod_{\tau\in \Sigma}T_{w_\tau}$) and we similarly define $\widehat{T}_{\tau,(0, 0)}$, $\widehat{T}_{(0, 0)}$, 
 $\widehat{T}_{w_{\tau}, (0, 0)}$ and $\widehat{T}_{w, (0, 0)}$.

For each $\l\in \{1, \ldots, k\}$, pick any $i\in S_l$, one may choose and fix a Galois character $\chi_l: G_{K}\map L^\times$ of  $\tau$-Hodge-Tate-Sen-weight $\in \wt_\tau(\delta_i)+\Z$ .




In our setting of general weights, we will need to work with $B_\dR$-representations that is almost de Rham after taking into account character twists.

\begin{definition}
Given characters $\chi_1, \ldots, \chi_k: K^\times \map L^\times$ that are not of integral $\tau$-weights difference to one another , we define the category $\mathrm{Rep}_{\pdr,\tau,  \chi_1, \ldots, \chi_k}(G_K)$ as the full subcategory of the category of finite free $B_\dR\otimes_{K, \tau}L$-representations of $G_K$, spanned by the representations $W$ such that 
\[
\dim_{B_\dR\otimes_{K, \tau}L}(W)=\sum_{l=1}^k \dim_{L}D_{\pdr, \tau, l}(W)
\]
Here we set $D_{\pdr, \tau, l}(W):=(B_\pdr\otimes_{B_\dR}W(\chi_l^{-1}))^{G_K}$. In particular, we define the category $\mathrm{Rep}_{\pdr,\tau}(G_K):=\mathrm{Rep}_{\pdr,\tau,  \chi_{\mathrm{triv}}}(G_K)$ to be the usual category of almost de Rham representation, $D_{\pdr, \tau}(W):=(B_\pdr\otimes_{B_\dR}W)^{G_K}$, with $\chi_{\mathrm{triv}}$ being the trivial character. Also we set $\rep_L(\mathbb{G}_a)$ to be the category of $L$-vector spaces equipped with a nilpotent linear operator.
\end{definition}

\begin{remark}\label{drfinitefree}
    For any object in the category of $G_K$-representations on finite generated $B_\dR\otimes_{K, \tau}L$-modules, it is automatically finite free since $G_K$ acts transitively on the factors in the decomposition $B_\dR\otimes_{K, \tau}L\cong \bigoplus_{\iota}B_\dR$, where $\iota: L\inj \overline{K}$ ranges through those embeddings that restricts to $\tau^{-1}$ on $\tau(K)$. We denote this category by $\mathrm{Rep}_{B_\dR, \tau}(G_K)$.
\end{remark}

\begin{lemma}\label{drequiv}
The functor $\bigoplus_{l=1}^k D_{\pdr, \tau, l}$ induces an equivalence of category between $\mathrm{Rep}_{\pdr, \tau, \chi_1, \ldots, \chi_k}(G_K)$ and $(\mathrm{Rep}_{L}(\mathbb{G}_a))^k$.
\end{lemma}

\begin{proof}
First, we show that $D_{\pdr, \tau}$ induces an equivalence of cagtegory between $\mathrm{Rep}_{\pdr,\tau}(G_K)$ and $\mathrm{Rep}_{L}(\mathbb{G}_a)$: We have the following commutative diagram
\[
\begin{tikzcd}
    \mathrm{Rep}_{\pdr,\tau}(G_K) \arrow[r, "D_{\pdr, \tau}"] \arrow[d] & \mathrm{Rep}_{L}(\mathbb{G}_a) \arrow[d] \\
    \mathrm{Rep}^L_{\pdr}(G_K) \arrow[r, "D_{\pdr}"] & \mathrm{Rep}_{L\otimes_{\Q_p} K}(\mathbb{G}_a)
\end{tikzcd}
\]
Here $\mathrm{Rep}^L_{\pdr}(G_K)$ is the category of almost de Rham $B_\dR$ of $G_K$ with an $L$-action, and $\mathrm{Rep}_{L\otimes_{\Q_p} K}(\mathbb{G}_a)$ is the category of $L\otimes_{\Q_p} K$-modules with a nilpotent linear operator $\nu$. By looking at the dimension formula in the definition of $\mathrm{Rep}_{\pdr,\tau}(G_K)$, we see that a $B_\dR\otimes_{K, \tau}L$-representations of $G_K$ is almost de Rham if and only if it is almost de Rham as a $B_\dR$-representation of $G_K$. Thus $\mathrm{Rep}_{\pdr,\tau}(G_K)$ is a direct factor of $\mathrm{Rep}^L_{\pdr}(G_K)$ by applying the idempotent $e_\tau\in L\otimes_{\Q_p} K$ giving the factor labelled $\tau$ in the decomposition $L\otimes_{\Q_p} K\cong \bigoplus_{\tau\in \Sigma}L$. Also $\mathrm{Rep}_{L}(\mathbb{G}_a)=e_\tau\mathrm{Rep}_{L\otimes_{\Q_p} K}(\mathbb{G}_a)$. Now the bottom row is an equivalence of category by \cite[3.1.1]{BHS}. Applying $e_\tau$ to it, we see that the top row is also an equivalence of category. 

By \cite[3.1.2]{BHS}, the bottom row has a quasi-inverse given by $(V, \nu)\mapsto W(V, \nu)$ (see \cite{BHS} for the notation here). Thus the same functor induces a quasi-inverse to the top row as well, still denoted by $W(V, \nu)$. Note that for any $(V, \nu)\in \mathrm{Rep}_{L}(\mathbb{G}_a)$, $\dim_{B_\dR\otimes_{K, \tau}L}(W(V, \nu))=\dim_L V$ and $\dim_{B_\dR}(W(V, \nu))=\dim_K V$.

Secondly, we claim that for any $W\in \mathrm{Rep}_{\pdr, \tau, \chi_1, \ldots, \chi_k}(G_K)$, we have $W\cong \bigoplus_l W(D_{\pdr, \tau, l}(W))(\chi_l)$. This gives a quasi-inverse to the functor $\bigoplus_{l=1}^k D_{\pdr, \tau, l}$ and thus concludes the proof. For each $l$, we first construct a canonical injection $W(D_{\pdr, \tau, l}(W))(\chi_l)\inj W$. In fact, there is an injection $D_{\pdr, \tau, l}(W)\otimes_K B_\pdr\cong (B_\pdr\otimes_{B_\dR}W(\chi_l^{-1}))^{G_K}\otimes_{K}B_\pdr\inj B_\pdr\otimes_{B_\dR}W(\chi_l^{-1})$ equivariant with respect to the nilpotent operator. And thus there is an injection $W(D_{\pdr, \tau, l}(W))\cong (D_{\pdr, \tau, l}(W)\otimes_K B_\pdr)^{\nu\otimes 1+1\otimes \nu_{B_\pdr}=0}\inj (B_\pdr\otimes_{B_\dR}W(\chi_l^{-1}))^{\nu_{B_\pdr}\otimes 1=0}\cong W(\chi_l^{-1})$. Twisting by $\chi_l$ gives the map. Furthermore, the induced map 
\[\bigoplus_l W(D_{\pdr, \tau, l}(W))(\chi_l)\map W
\]
is an injection. To see this, it suffices to show that for any $l$, there is no nontrivial intersection between the image of  $W(D_{\pdr, \tau, l}(W))(\chi_l)$ and the image of $\bigoplus_{l'\neq l} W(D_{\pdr, \tau, l'}(W))(\chi_{l'})$. If there is, by the exactness of the functor $D_{\pdr, \tau, l}$ \cite[3.17]{Fon}, we deduce that $D_{\pdr, \tau, l}(\bigoplus_{l'\neq l} W(D_{\pdr, \tau, l'}(W))(\chi_{l'}))\neq 0$, i.e. one have $D_{\pdr}(\bigoplus_{l'\neq l} W(D_{\pdr, \tau, l'}(W))(\chi_{l'}\chi_l^{-1}))\neq 0$. However, each factor 
\[D_\pdr (W(D_{\pdr, \tau, l'}(W))(\chi_{l'}\chi_l^{-1}))=0\] because \[\dim_K D_\pdr (W(D_{\pdr, \tau, l'}(W))(\chi_{l'}\chi_l^{-1}))\leq \dim_K D_{\mathrm{pHT}}(W(D_{\pdr, \tau, l'}(W))(\chi_{l'}\chi_l^{-1}))=0\]
since the latter has $\tau$-Hodge-Tate-Sen-weights not in $\Z$ by the assumption that $\chi_{l'}$ and $\chi_l$ are not of integral $\tau$-weights difference to each other.

Now that we have an injection $\bigoplus_l W(D_{\pdr, \tau, l}(W))(\chi_l)\inj W$, we check that the $B_\dR$-dimension of the left hand side is $[L: K]\cdot\sum_{l=1}^k \dim_{L}D_{\pdr, \tau, l}(W)$, which is equal to $\dim_{B_\dR}(W)$, by our assumption. Thus the injection is an isomorphism and the proof concludes.

\end{proof}

\begin{lemma}\label{stabledr}
The category $\mathrm{Rep}_{\pdr,\tau,  \chi_1, \ldots, \chi_k}(G_K)$ as a full subcategory of the category of finite free $B_\dR\otimes_{K, \tau}L$-representations of $G_K$, is stable  under taking sub, quotient and extension.
\end{lemma}

\begin{proof}
The stability under extension is immediate by using the exactness of each $D_{\pdr, \tau, l}$. For $W\in \mathrm{Rep}_{\pdr,\tau,  \chi_1, \ldots, \chi_k}(G_K)$ and $W'\subset W$ a finite free $B_\dR\otimes_{K, \tau}L$-submodule stable under $G_K$, we show that $W'\cong \bigoplus_{l=1}^k\mathrm{pr}_l(W')$, where $\mathrm{pr}_l: W\map W(D_{\pdr, \tau, l}(W))(\chi_l)$ is the $l$-th projection in the decomposition $W\cong \bigoplus_l W(D_{\pdr, \tau, l}(W))(\chi_l)$ as in the proof of Lemma \ref{drequiv}. Let $W_l:=W(D_{\pdr, \tau, l}(W))(\chi_l)$ and $i_l: W_l \inj W$ be the natural inclusion of the $l$-th factor in the above decomposition. Then it suffices to show that the inclusion $W'\cap i_l(W_l)\inj \mathrm{pr}_l(W')$ is an equality. We have the short exact sequence
\[
0 \map W'\cap i_l(W_l) \map W' \map \left(\bigoplus_{l'\neq l}p_{l'}\right)(W') \map 0
\]
. Using the exactness of $D_{\pdr, \tau, l}$ one sees $D_{\pdr, \tau, l}(W'\cap i_l(W_l))=D_{\pdr, \tau, l}(W')$ since the last term of the above short exact sequence is contained in $\bigoplus_{l'\neq l}W_{l'}$ and  $D_{\pdr, \tau, l}(\bigoplus_{l'\neq l}W_{l'})=0$ as in the proof of Lemma \ref{drequiv}. Similarly, the short exact sequence
\[
0 \map (\bigoplus_{l'\neq l}W_{l'})\cap W' \map W' \map \mathrm{pr}_l(W') \map 0
\]
gives $D_{\pdr, \tau, l}(\mathrm{pr}_l(W'))=D_{\pdr, \tau, l}(W')$. Now $\mathrm{Rep}_{\pdr,\tau,  \chi_l}(G_K)$ is isomorphic to $\mathrm{Rep}_{\pdr,  \chi_l}(G_K)$ by twisting $\chi_l^{-1}$, the latter being stable under sub and quotient implies  $W'\cap i_l(W_l)$ and   $\mathrm{pr}_l(W')$ are both in $\mathrm{Rep}_{\pdr,  \chi_l}(G_K)$.   By twisted version of \cite[3.1.1]{BHS} (or Lemma \ref{drequiv}), we see that $D_{\pdr, \tau, l}$ induces an equivalence between $\mathrm{Rep}_{\pdr,\tau,  \chi_l}(G_K)$ and $\rep_L(\mathbb{G}_a)$. Since $D_{\pdr, \tau, l}(W'\cap i_l(W_l))\cong D_{\pdr, \tau, l}(W') \cong D_{\pdr, \tau, l}(\mathrm{pr}_l(W'))$ induced by the natural map $W'\cap i_l(W_l)\inj \mathrm{pr}_l(W')$ is an isomorphism, we conlude that $W'\cap i_l(W_l)\cong \mathrm{pr}_l(W')$. And thus $W'$ can be written under in the form $\bigoplus_{l=1}^k\mathrm{pr}_l(W')$. Each  $\mathrm{pr}_l(W')\subset W_l$ is an object in $\mathrm{Rep}_{\pdr,\tau,  \chi_l}(G_K)$, so we see $W'\in \mathrm{Rep}_{\pdr,\tau,  \chi_1, \ldots, \chi_k}(G_K)$. Any quotient of $W'$ of $W$ has the form $\bigoplus_{l=1}^kW_l'$ where each $W_l'$ is a quotient of $W_l$, and thus in $\mathrm{Rep}_{\pdr,\tau,  \chi_l}(G_K)$.
\end{proof}

Let $\mathcal{C}_L$ be the category of finite dimensional local Artinian $L$-algebra with residue field $L$.

\begin{definition}
Let $A\in \mathcal{C}_L$. 
\begin{enumerate}
    \item We define $\mathrm{Rep}_{\pdr, A, \tau, \chi_1, \ldots, \chi_k}$ to be the full subcategory of the category of finite free $B_\dR\otimes_{K, \tau} A$-represenations of $G_K$, spanned by the represenations $W$ such that $\dim_{B_\dR\otimes_{K, \tau}L}(W)=\sum_{l=1}^k \dim_{L}D_{\pdr, \tau, l}(W)$. We also define $\rep_{A}(\mathbb{G}_a)$ to be the category of finite free $A$-modules equipped with a nilpotent $A$-linear operator. We will write the nilpotent opeartor on $\bigoplus_l D_{\pdr, \tau, l}(W)$ as $\nu_{W}=\bigoplus_l \nu_{W, l}$.
    \item A filtered $B_\dR\otimes_{K, \tau} A$-represenation $(W, \mathcal{F}_\bullet)$ of $G_K$ is an $B_\dR\otimes_{K, \tau} A$-represenation $W$ of $G_K$ of rank $n$ with an increasing filtration $(\mathcal{F}_i)_{i\in \{1, \ldots, n\}}$ by $B_\dR\otimes_{K, \tau} A$-represenation of $G_K$, such that all graded pieces $\mathcal{F}_i/\mathcal{F}_{i-1}$ is finite free of rank $1$ over $B_\dR\otimes_{K, \tau} A$ for any $i\in \{1, \ldots, n\}$.
    \end{enumerate}
\end{definition}

\begin{remark}
\begin{enumerate}
    \item By the exactness of the functor $D_{\pdr, \tau, l}$ and a devissage argument, we see that any finite free $B_\dR\otimes_{K, \tau} A$-represenations of $G_K$ that deforms a representation $W/\mathfrak{m}_A W\in \mathrm{Rep}_{\pdr, \tau, \chi_1, \ldots, \chi_k}$ automatically lies in $\mathrm{Rep}_{\pdr, A, \tau, \chi_1, \ldots, \chi_k}$.
    \item By Lemma \ref{stabledr}, for any filtered $B_\dR\otimes_{K, \tau} A$-represenation $(W, \mathcal{F}_\bullet)$ of $G_K$, if $W\in \mathrm{Rep}_{\pdr, A, \tau, \chi_1, \ldots, \chi_k}$, then the subquotient $\mathcal{F}_j/\mathcal{F}_i\in \mathrm{Rep}_{\pdr, A, \tau, \chi_1, \ldots, \chi_k}$, for any $j\geq i$.
\end{enumerate}
\end{remark}

\begin{lemma}\label{draeq}
The functor $\bigoplus_l D_{\pdr, \tau, l}$  induces an equivalence of category between  $\mathrm{Rep}_{\pdr, A, \tau, \chi_1, \ldots, \chi_k}$ and $(\rep_{A}(\mathbb{G}_a))^k$.
\end{lemma}

\begin{proof}
The proof is very similar to that of \cite[3.1.4]{BHS}. By Lemma \ref{drequiv}, it suffices to check that for any $W$ a $G_K$-representation on $B_\dR\otimes_{K, \tau}A$ module, it is finite free as a $B_\dR\otimes_{K, \tau}A$-module if and only if $\bigoplus_l D_{\pdr, \tau, l}(W)$ is finite free as $A$-module. Now for a finitely generated $B_\dR\otimes_{K, \tau}A$-module $M$  with $G_K$-action or an $A$-module, it is finite free if and only if it is flat as an $A$-module (as $M/\mathfrak{m}_AM$ is finite free over $B_\dR\otimes_{K, \tau} L$ by Remark \ref{drfinitefree}). Thus it suffices to show that $W$ is flat as an $A$-module if and only if $\bigoplus_l D_{\pdr, \tau, l}(W)$ is flat as an $A$-module. Let $N$ be a finite $A$-module. By writing it in the form $A^m/A^n$, and using the exactness of the functor $\bigoplus_l D_{\pdr, \tau, l}$ (on $\mathrm{Rep}_{\pdr, \tau, \chi_1, \ldots, \chi_k}(G_K)$), we see that $M\otimes_A \left(\bigoplus_l D_{\pdr, \tau, l}(W)\right)\cong \bigoplus_l D_{\pdr, \tau, l}(M\otimes_AW)$ and that $M\otimes_AW\in \mathrm{Rep}_{\pdr, \tau, \chi_1, \ldots, \chi_k}(G_K)$. One see that $\bigoplus_l D_{\pdr, \tau, l}$ preserves flatness as $A$-module and vice-versa.
\end{proof}

\begin{definition}\label{xwdef}
Fix $W\in \mathrm{Rep}_{\pdr, \tau, \chi_1, \ldots, \chi_k}$ and fix $\alpha=\bigoplus_l\alpha_l: \bigoplus_l L^{n_l} \cong  \bigoplus_l D_{\pdr, \tau, l}(W)$. Define $X_{W, \tau}^\square$ to be the groupoid over $\mathcal{C}_L$ consisting of objects $(A, W_A, \iota_A, \alpha_A)$ (and obvious morphisms) where $W_A\in \mathrm{Rep}_{\pdr, A, \tau, \chi_1, \ldots, \chi_k}$, $\iota_A: W_A\otimes_A L\cong W$ and $\alpha_A=\bigoplus_{l=1}^k \alpha_{A, l}$, where each $\alpha_{A, l}: A^{n_l}\cong D_{\pdr, \tau, l}(W_A)$ such that the following diagram commutes:
\[
\begin{tikzcd}[column sep=large]
L^{n_l}\arrow[r, "\alpha_{A, l}\  \mathrm{mod}\  \mathfrak{m}_A"] \arrow[d, "="] &L\otimes_A  D_{\pdr, \tau, l}(W_A) \arrow[d, "\cong"]\\
L^{n_l} \arrow[r, "\alpha"] & D_{\pdr, \tau, l}(W)
\end{tikzcd}
\]
for any $l\in \{1, \ldots, k\}$. Similarly we define $X_{W, \tau}$ as above but without framing.
\end{definition}

\begin{cor}\label{xwchar}
Notation as above. The groupoid $X^\square_{W, \tau}$ is pro-representable. The functor:
\[
(W_A, \iota_A, \alpha_A) \mapsto N_{W_A}
\] 
where we set $N_{W_A}:= \bigoplus_{l=1}^k N_{W_A, l}$ to be the matrix of the nilpotent operator $\alpha_A^{-1}\circ \nu_{W_A}\circ \alpha_A$ on $\bigoplus_{l=1}^k A^{n_l}$ under the canonical basis, induces an equivalence between $|X^\square_{W, \tau}|$ and $\prod_{l=1}^k \widehat{\mathfrak{g}}_{n_l} $, where the latter denotes the completion of $\prod_{l=1}^k \mathfrak{g}_{n_l}$ at $N_W$, viewed as a functor $\mathcal{C}_L\map \mathrm{Sets}$.
\end{cor}

\begin{proof}
Immediate from Lemma \ref{draeq}.
\end{proof}

\begin{definition}
Let $W\in \mathrm{Rep}_{\pdr, \tau, \chi_1, \ldots, \chi_k}$ and $(W, \mathcal{F}_\bullet)$ be a filtered $B_\dR\otimes_{K, \tau} L$-represenation  of $G_K$.  We define $X_{W, \mathcal{F}_\bullet}^\square$ to be  a groupoid over $\mathcal{C}_L$, whose objects are $(A, W_A, \mathcal{F}_{A, \bullet}, \iota_A, \alpha_A)$ where $(W_A, \mathcal{F}_{A, \bullet})$ is a filtered $B_\dR\otimes_{K, \tau} A$-represenation of $G_K$, $\iota_A: W_A\otimes_A L\cong W$ an isomorphism inducing $\mathcal{F}_{A, i}\otimes_A L\cong \mathcal{F}_i$ for all $i$, and $(W_A, \iota_A, \alpha_A)\in X_W^\square(A)$.  Similarly we define $X_{W, \mathcal{F}_\bullet}$ as above but without framing.
\end{definition}

For each $l\in \{1, \ldots, k\}$, we set $\mathcal{D}_{A, \tau, l, \bullet}$ (resp. $\mathcal{D}_{\tau, l, \bullet}$) be the increasing filtration on $D_{\pdr, \tau, l}(W_A)$ (resp. $D_{\pdr, \tau, l}(W)$) induced by $D_{\pdr, \tau, l}(\mathcal{F}_{A, \bullet})$ (resp.  $D_{\pdr, \tau, l}(\mathcal{F}_{\bullet})$). By Lemma \ref{stabledr}, we see that $\mathcal{D}_{A, \tau, l, \bullet}$ (resp. $\mathcal{D}_{\tau, l, \bullet}$) gives a complete flag on $D_{\pdr, \tau, l}(W_A)$ (resp. $D_{\pdr, \tau, l}(W)$). The proof of Lemma \ref{stabledr} also shows that $\mathcal{D}_{A, \tau, l, i}/\mathcal{D}_{A, \tau, l, i-1}$ (resp. $\mathcal{D}_{\tau, l, i}/\mathcal{D}_{\tau, l, i-1}$) are rank $1$ over $A$ (resp. $L$) if and only if $i\in S_l$, where $S_l$ is the subset of $\{1, \ldots, n\}$ consists of $i$ such that $\mathcal{F}_i/\mathcal{F}_{i-1}$ is of Hodge-Tate-Sen weight in the same integer difference class with $\chi_l$. These filtrations are stable under $\nu_{W_A, l}$ (resp. $\nu_{W, l}$). We denote by $\widehat{\widetilde{\mathfrak{g}}}_{n_l}$ the completion of $\widetilde{\mathfrak{g}}_{n_l}$ at the $L$-point $(\alpha_{ l}^{-1}(\mathcal{D}_{\tau, l, \bullet}), N_{W, l})$. (See Definition \ref{xwdef} and Corollary \ref{xwchar} for notations.)

\begin{cor}\label{xwf}
The groupoid $X_{W, \mathcal{F}_\bullet}^\square$ over $\mathcal{C}_L$ is pro-representable. The functor:
\[
(W_A, \mathcal{F}_{A, \bullet}, \iota_A, \alpha_A)\mapsto \prod_{l=1}^k(\alpha_{A, l}^{-1}(\mathcal{D}_{A,\tau, l, \bullet}), N_{W_A, l})
\]
induces an isomorphism between $|X_{W, \mathcal{F}_\bullet}^\square|$ and $\prod_{l=1}^k\widehat{\widetilde{\mathfrak{g}}}_{n_l}$.
\end{cor}

\begin{proof}
Immdediate from Corollary 3.11 and Lemma \ref{stabledr}.
\end{proof}

\begin{definition}\label{def313}
\begin{enumerate}
    \item  Given characters $\chi_1, \ldots, \chi_k: K^\times \map L^\times$ that are not of integral $\tau$-weights difference to one another , we define the category $\mathrm{Rep}^+_{\pdr,\tau,  \chi_1, \ldots, \chi_k}(G_K)$ as the full subcategory of the category of finite free $B^+_\dR\otimes_{K, \tau}L$-representations of $G_K$, spanned by the representations $W^+$ such that $W:=W^+\invertt\in \mathrm{Rep}_{\pdr,\tau,  \chi_1, \ldots, \chi_k}(G_K)$. Let $A\in \mathcal{C}_L$. Define $\mathrm{Rep}^+_{\pdr, A, \tau, \chi_1, \ldots, \chi_k}$ to be the full subcategory of the category of finite free $B^+_\dR\otimes_{K, \tau} A$-represenations of $G_K$, spanned by the represenations $W^+$ such that $W^+\in \mathrm{Rep}^+_{\pdr,\tau,  \chi_1, \ldots, \chi_k}(G_K)$ as a finite free $B^+_\dR\otimes_{K, \tau}L$-representations of $G_K$.
    \item We set $\Fil\rep_L(\mathbb{G}_a)$(resp. $\Fil\rep_A(\mathbb{G}_a)$) to be the category of (decreasingly) filtered $L$(resp. $A$)-vector spaces equipped with a nilpotent linear operator that preserves the filtration, such that the graded pieces are finite free.
    \item For any $W\in \mathrm{Rep}_{\pdr,\tau,  \chi_1, \ldots, \chi_k}(G_K)$(resp. $\in \mathrm{Rep}_{\pdr, A, \tau, \chi_1, \ldots, \chi_k}$), we say $W^+\in \mathrm{Rep}^+_{\pdr,\tau,  \chi_1, \ldots, \chi_k}(G_K)$ (resp. $\in \mathrm{Rep}^+_{\pdr, A, \tau, \chi_1, \ldots, \chi_k}$) is a lattice of $W$ if there is an identification $W\cong W^+\invertt$. For any $B^+_\dR$-lattice $W^+$ inside $W$, we define $\Fil^i_{W^+, \tau, l}(D_{\pdr, \tau, l}(W)):= (t^iB^+_{\pdr}\otimes_{B^+_{\dR}}W^+(\chi_l^{-1}))^{G_K}$.
\end{enumerate}
 
\end{definition}

 We have the following analogue to \cite[3.2.1]{BHS}

\begin{lemma}
\label{drplusequiv}
For any $W\in \mathrm{Rep}_{\pdr, \tau, \chi_1, \ldots, \chi_k}(G_K)$ , the map 
\[
W^+\mapsto \bigoplus_l\Fil^\bullet_{W^+,\tau, l}(D_{\pdr, \tau, l}(W))
\]
induces a bijection between the set of $G_K$-stable  $B^+_{\dR}\otimes_{K, \tau}L$-lattices of $W$ and the set of filtrations on each of the $D_{\pdr, \tau, l}(W)$ as  $\mathbb{G}_a$-representations.

\end{lemma}

\begin{proof}
Following the proof of Lemma \ref{drequiv}, we write $W\cong \bigoplus_l W_l$, where $W_l:=\bigoplus_l W(D_{\pdr, \tau, l}(W))(\chi_l)$. And we let $i_l$ and $p_l$ be the inclusion and projection map of the $l$-th factor.

First, we show that any  $G_K$-stable  $B^+_{\dR}\otimes_{K, \tau}L$-lattices $W^+$ of $W$ has the form $\bigoplus_l W^+_l$, where each $W^+_l$ is a $G_K$-stable  $B^+_{\dR}\otimes_{K, \tau}L$-lattices of $W_l$. We only need to show $i_l(W_l)\cap W^+=p_l(W^+)$ for any $l\in \{1, \ldots, k\}$. We have a short exact sequence of finite free $B^+_\dR$-module 
\[
0 \map i_l(W_l)\cap W^+ \map W^+ \map \left(\bigoplus_{l'\neq l}p_{l'}\right)(W^+) \map 0
\]
and thus a left exact sequence
\[
0 \map \Fil^i_{W^+, \tau, l}(i_l(W_l)\cap W^+) \map \Fil^i_{W^+, \tau, l}(W^+) \map \Fil^i_{W^+, \tau, l}\left(\left(\bigoplus_{l'\neq l}p_{l'}\right)(W^+)\right)
\]
the last term is $0$ as it is contained in $D_{\pdr, \tau, l}(W_{l'})$, which is $0$ by the proof of Lemma \ref{drequiv}. Thus $\Fil^i_{W^+, \tau, l}(i_l(W_l)\cap W^+) \cong \Fil^i_{W^+, \tau, l}(W^+)$ for any $i\in \Z$. We also have a short exact sequence of finite free $B^+_\dR$-module
\[
0 \map W^+\cap \bigoplus_{l'\neq l}W_{l'}\map W^+\map p_l(W^+) \map 0
\] 
and we deduce similarly $\Fil^i_{W^+, \tau, l}(p_l(W^+))=\Fil^i_{W^+, \tau, l}(W^+)$, noting that
\[H^1(G_K, t^iB^+_{\pdr}\otimes_{B^+_{\dR}}(W^+\cap \bigoplus_{l'\neq l}W_{l'})(\chi_l^{-1}))=0\]
since it can be filtered by a filtration whose graded pieces are $H^1(G_K, C[\log t]\otimes_C V)=0$ for $V$ a $C$-reprentation of $G_K$ having non-integral Hodge-Tate-Sen weights. Now $i_l(W_l)\cap W^+\subset p_l(W^+)$ are two $B^+_{\dR}\otimes_{K, \tau}L$-lattices of $W_l$, such that $\Fil^\bullet_{W^+,\tau, l}(i_l(W_l)\cap W^+)=\Fil^\bullet_{W^+,\tau, l}(p_l(W^+))$. Apply \cite[3.2.1]{BHS},  (twisted by $\chi_l$), we see that $i_l(W_l)\cap W^+= p_l(W^+)$.

Finally, we have seen $W^+\cong \bigoplus_l W^+_l$, where each $W^+_l$ is a $G_K$-stable  $B^+_{\dR}\otimes_{K, \tau}L$-lattices of $W_l$. Apply \cite[3.2.1]{BHS}  (twisted, with added $L$-action) to each $W_l$, we conclude the proof.
\end{proof}

\begin{lemma}\label{drplusaeq}
let $A\in \mathcal{C}_L$.   The functor defined by $W^+\mapsto \bigoplus_l\Fil^\bullet_{W^+,\tau, l}(D_{\pdr, \tau, l}(W))$ defines a bijection between $\mathrm{Rep}^+_{\pdr, A, \tau, \chi_1, \ldots, \chi_k}$ and $(\Fil\rep_A(\mathbb{G}_a))^k$.
\end{lemma}

\begin{proof}
For any $W^+\in \mathrm{Rep}^+_{\pdr, A, \tau, \chi_1, \ldots, \chi_k}$, we have by the proof of Lemma \ref{drplusequiv} $W^+\cong \bigoplus_l W_l^+$, where each $W_l^+$ is a $B_\dR^+\otimes_{K, \tau}L$ lattice of $W_l$ such that $W_l(\chi_l^{-1})$ is almost de Rham. By functorialty we see that each $W_l^+$ is in fact a $B_\dR^+\otimes_{K, \tau}A$-module. We claim it is finite free over $B_\dR^+\otimes_{K, \tau}A$: First $W_l^+$ is flat as an $A$-module since it is a direct summand of $W^+$, a flat $A$-module. Secondly, $W^+ /\mathfrak{m}_A\cong \bigoplus_l W^+_l/\mathfrak{m}_A$ gives the corresponding decomposition for $W^+ /\mathfrak{m}_A\in \mathrm{Rep}^+_{\pdr, \tau, \chi_1, \ldots, \chi_k}$. Thus by the proof of Lemma \ref{drplusequiv}, $W^+_l/\mathfrak{m}_A$ is finite free over module over $B_\dR^+\otimes_{K, \tau}L$. Combining the two facts we see $W_l^+$ is finite free over $B_\dR^+\otimes_{K, \tau}A$. Now apply \cite[3.2.2]{BHS}  (twisted) to each $l$ we conclude immediately.
\end{proof}

\begin{definition}
Fix $W^+\in \mathrm{Rep}^+_{\pdr, \tau, \chi_1, \ldots, \chi_k}$ and fix $\alpha=\bigoplus_l\alpha_l: \bigoplus_l L^{n_l} \cong  \bigoplus_l D_{\pdr, \tau, l}(W^+\invertt)$. Define $X_{W^+, \tau}^\square$ to be the groupoid over $\mathcal{C}_L$ consisting of objects $(A, W^+_A, \iota_A, \alpha_A)$ (and obvious morphisms) where $W^+_A\in \mathrm{Rep}^+_{\pdr, A, \tau, \chi_1, \ldots, \chi_k}$, $\iota_A: W^+_A\otimes_A L\cong W^+$ and $\alpha_A$ as in Definition \ref{xwdef} for $W_A:=W^+_A\invertt$. Similarly we define $X_{W^+, \tau}$ as above but without framing.
\end{definition}

\begin{definition}
Let $W^+\in \mathrm{Rep}^+_{\pdr, \tau, \chi_1, \ldots, \chi_k}$. We say it is $\tau$-regular if all Hodge-Tate-Sen weights of $W^+/tW^+$ are distinct from each other. This is equivalent to the condition that for any $l\in \{1, \ldots, k\}$, the graded pieces $\gr^i(D_{\pdr, \tau, l}(W))$ are all of dimension $\leq 1$ over $L$.
\end{definition}

Let $W^+\in \mathrm{Rep}^+_{\pdr, \tau, \chi_1, \ldots, \chi_k}$ be $\tau$-regular. Denote by $-h_{\tau,l, 1}> \cdots> -h_{\tau, l, n_l}$ the integers $i$ such that  $\gr^i(D_{\pdr, \tau, l}(W))\neq 0$, for any $l\in \{1, \ldots, k\}$. Let $A\in \mathcal{C}_L$ and $(W^+_A, \iota_A, \alpha_A)$ be an object of $X_{W^+}^\square(A)$, Lemma \ref{drplusaeq} gives a filtration $\Fil^\bullet_{W_A^+,\tau, l}$ on each of $D_{\pdr, \tau, l}(W_A)$. It follows from \cite[3.2.3]{BHS} that $\gr^i(D_{\pdr, \tau, l}(W_A))\otimes_A L\cong \gr^i(D_{\pdr, \tau, l}(W))$. And hence $\gr^i(D_{\pdr, \tau, l}(W_A))$ is a finite free $A$-module of rank $1$ when $i\in \{-h_{\tau,l, 1}, \cdots, -h_{\tau, l, n_l}\}$ and is $0$ otherwise. We can define a complete (increasing) flag $\Fil_{W^+_A, \tau, l, \bullet}$ on each $D_{\pdr, \tau, l}(W_A)$ by setting 
\[
\Fil_{W^+_A, \tau, l, i}(D_{\pdr, \tau, l}(W_A)):= \Fil^{-h_{\tau, l, i}}(D_{\pdr, \tau, l}(W_A))
\]
for any $i\in \{1, \ldots, n_l\}$. The filtration is stable under $\nu_{W_A, l}$. Thus the $k$ pairs $\bigoplus_{l=1}^k (\alpha_{A, l}^{-1}(\Fil_{W^+_A, \tau, l, \bullet}), N_{W_A, l})$ defines an element in  $\prod_{l=1}^k\widetilde{\mathfrak{g}}_{n_l}(A)$. Furthermore, let $\prod_{l=1}^k\widehat{\widetilde{\mathfrak{g}}}_{n_l}$ denote the completion of $\prod_{l=1}^k\widetilde{\mathfrak{g}}_{n_l}$ at the $L$ point given by $k$ pairs $\bigoplus_{l=1}^k (\alpha_{ l}^{-1}(\Fil_{W^+, \tau, l, \bullet}), N_{W, l})$, then the above construction in fact gives a point in $\prod_{l=1}^k\widehat{\widetilde{\mathfrak{g}}}_{n_l}(A)$.

\begin{cor}\label{xwplus}
Let $W^+\in \mathrm{Rep}^+_{\pdr, \tau, \chi_1, \ldots, \chi_k}$ be $\tau$-regular. The groupoid $X^\square_{W^+, \tau}$ is pro-representable. The functor:
\[
(W^+_A, \iota_A, \alpha_A)\mapsto \bigoplus_{l=1}^k (\alpha_{A, l}^{-1}(\Fil_{W^+_A, \tau, l, \bullet}), N_{W_A, l})
\]
induces an isomorphism of functors between $|X^\square_{W^+, \tau}|$ and $\prod_{l=1}^k\widehat{\widetilde{\mathfrak{g}}}_{n_l}$, treating the latter as a functor $\mathcal{C}_L\map \mathrm{Sets}$ again.
\end{cor}

\begin{proof}
Immediate from Lemma \ref{drplusaeq}.
\end{proof}

Let  $A\in \mathcal{C}_L$ and $\mathcal{M}$ (resp. $D$) be a trianguline $\phigamma$-module over $\rob_{A, K}\invertt$ (resp. $\rob_{A, K}$) of rank $n$ of parameters $\delta_1, \ldots, \delta_n$. Fix $\tau\in \Sigma$. Decompose $\{1, \ldots, n\}$ according to the equivalence relation of integral $\tau$-weights differences of $\overline{\delta}_i$. We obtain as in the beginning of the section $k$ equivalence classes and characters $\chi_1, \ldots, \chi_k$. Recall that in \cite[3.3]{BHS}, $W_{\dR}(\mathcal{M})$ (resp. $W_{\dR}^+(D)$) is defined to be a $G_K$-representaion over $B_\dR\otimes_{\Q_p}A\cong \bigoplus_{\tau\in \Sigma}B_{\dR}\otimes_{K, \tau}A$ (resp. $B_\dR^+\otimes_{\Q_p}A\cong \bigoplus_{\tau\in \Sigma}B_{\dR}^+\otimes_{K, \tau}A$), we let $W_{\dR, \tau}(\mathcal{M})$ (resp. $W_{\dR, \tau}^+(D)$) be the factor of $W_{\dR}(\mathcal{M})$ (resp. $W_{\dR}^+(D)$) corresponding to the embedding $\tau$.

\begin{lemma}\label{imageright}
Let $\mathcal{M}$ as above. Then $W_{\dR, \tau}(\mathcal{M})\in \mathrm{Rep}_{\pdr, A, \tau, \chi_1, \ldots, \chi_k}$ and is free of rank $n$ over $B_{\dR}\otimes_{K, \tau}A$.
\end{lemma}

\begin{proof}
By \cite[3.3.5]{BHS}, we have that $W_{\dR}(M)$ is finite free over $B_\dR\otimes_{\Q_p}A$ of rank $n$, hence $W_{\dR, \tau} (\mathcal{M})$  free of rank $n$ over $B_{\dR}\otimes_{K, \tau}A$. To see that it is in $\mathrm{Rep}_{\pdr, A, \tau, \chi_1, \ldots, \chi_k}$, by a devissage argument, using the fact $D_{\pdr, \tau, l}$ is exact, it suffices to prove the claim for $\mathcal{M}=\rob_{A, k}(\delta_i)$. Another devissage reduces to the case $A=L$. In this case, suppose the character $\delta_i$ has $\tau$-Hodge-Tate-Sen weight in the $\tau$-weight class labelled by $l$. Then one have $W_{\dR, \tau}(\mathcal{M})=B_{\dR}\otimes_{K, \tau}A(\delta_i)$, $D_{\pdr, \tau, l'}(W_{\dR, \tau}(\mathcal{M}))=0$ for any $l'\neq l$ and $D_{\pdr, \tau, l}(W_{\dR, \tau}(\mathcal{M}))=0$ is of dimension $1$ over $L$. We have $\dim_{B_\dR\otimes_{K, \tau}L}(W_{\dR, \tau}(\mathcal{M}))=\sum_{l=1}^k \dim_{L}D_{\pdr, \tau, l}(W_{\dR, \tau}(\mathcal{M}))$ in this case and we are done.
\end{proof}

Given $\mathcal{M}$ (resp. $D$) a trianguline $\phigamma$-module over $\rob_{L, K}\invertt$ (resp. $\rob_{L, K}$) of rank $n$, whose triangulation denoted by $\mathcal{M}_\bullet$,  we define the groupoid $X_{\mathcal{M}}$, $X_D$, $X_{\mathcal{M}, \mathcal{M}_\bullet}$ and the map of groupoids $\omega_{\underline{\delta}}: X_{\mathcal{M}, \mathcal{M}_\bullet}\map \widehat{\mathcal{T}^n_{\underline{\delta}}}$ as in \cite[Page 36]{BHS}. Let $W_\tau:=W_{\dR, \tau}(\mathcal{M})$, and $\mathcal{F}_{\tau, \bullet}:=W_{\dR, \tau}(\mathcal{M}_\bullet)$. Applying $W_{\dR, \tau}$ to each member of the filtration $\mathcal{M}_{A, \bullet}$, using its exactness (which follows from proof of \cite[3.3.5]{BHS}) and Lemma \ref{imageright}, we have a map of groupoids $X_{\mathcal{M}, \mathcal{M}_\bullet} \map X_{W_\tau, \mathcal{F}_{\tau, \bullet}}$. 

\begin{cor}\label{commdiag1}
The diagram of groupoids
\[
\begin{tikzcd}
X_{\mathcal{M}, \mathcal{M}_\bullet} \arrow[r] \arrow[d, "\omega_{\underline{\delta}}"]   & X_{W_\tau, \mathcal{F}_{\tau, \bullet}}\arrow[d, "\kappa_{\tau, W_\tau, \mathcal{F}_{\tau, \bullet}}"]\\
\widehat{\mathcal{T}^n_{\underline{\delta}}} \arrow[r, "\wt_\tau-\wt_\tau(\underline{\delta})"] & \bigoplus_{l=1}^k\widehat{\mathfrak{t}}_{\tau, l}
\end{tikzcd}
\]
is commutative. Here, the bottom map takes the following form: it decomposed $\bigoplus_{i=1}^n \delta_{A, i}\in \widehat{\mathcal{T}^n_{\underline{\delta}}}(A)$ into $\bigoplus_{l=1}^k(\delta_{A, i})_{i\in S_l}$ where each $(\delta_{A, i})_{i\in S_l}$ is listed in increasing order of $i\in s_l$, and then apply the map $\wt_\tau-\wt_\tau(\delta_i)_{i\in S_l}$ to each $(\delta_{A, i})_{i\in S_l}$ with image in $\widehat{\mathfrak{t}}_{\tau, l}$.
\end{cor}

\begin{proof}
The twisted analogue of \cite[3.3.6]{BHS} carries over.
\end{proof}

\subsection{A Formally Smooth Morphism}
Proposition \ref{formsm1} is the key property to prove the existence of a local model in the next subsection.

We need the following more general version of \cite[3.3.3]{BHS}

\begin{lemma}\label{modt}
Let $\mathbf{k}=(k_\tau)_{\tau\in \Sigma} \in \Z^{[K:\Q_p]}_{\geq 0}$, $\delta: K^\times\map L^\times$ a continuous character, $j\in \{0, 1\}$ and $S\subset \Sigma$ a a subset. Assume $\wt_{\tau}(\delta)\in \{1-k_\tau, \ldots, 0\}$ if and only if $\tau\in S$. Then   $\dim_L H^j_{\hphigamma}(\rlk(\delta)/t^{\mathbf{k}})=\mathrm{Card} (S)$.
\end{lemma}

\begin{proof}
Induction from \cite[Proposition 2.14]{Bergdall} .
\end{proof}

\begin{lemma}\label{h0prop}
Let $\delta$ be a continuous character $K^\times\map A^\times$ and $\overline{\delta}: K^\times \map L^\times$ be its reduction. Assume $\overline{\delta}$ that is not algebraic. 
\begin{enumerate}
    \item We have $H^0_{\hphigamma}(\rob_{A, K}(\delta)\invertt)=0$.
    \item  If $\wt_{\tau}(\overline{\delta})\notin\Z_{>0}$ for any $\tau\in \Sigma$, then $H^0_{\hphigamma}(\rob_{A, K}(\delta)\invertt/\rob_{A, K}(\delta))=H^1_{\hphigamma}(\rob_{A, K}(\delta)\invertt/\rob_{A, K}(\delta))=0$.
\end{enumerate}
\end{lemma}

\begin{proof}
A devissage argument reduce the case to $A=L$.  Observe that $H^i_{\hphigamma}(D\invertt)=\mathrm{colim}_k H^i_{\hphigamma}(t^{-k}D)$, and $H^i_{\hphigamma}(D\invertt/D)=\mathrm{colim}_k H^i_{\hphigamma}(t^{-k}D/D)$ for any $\phigamma$-module $D$ over $\rob_{L, K}$. We conclude using \cite[3.3.3]{BHS} and \cite[3.4.1]{BHS}.
\end{proof}

\begin{lemma}\label{surjforsm}
Let $\delta: K^\times\map L^\times$ be any continuous character such that $\delta$ and $\epsilon\delta^{-1}$ are not algebraic. Then the natural map
\[
H^1_{\hphigamma}(\rob_{L, K}(\delta)\invertt) \map H^1(G_K, W_{\dR}(\rlk(\delta)\invertt))
\]
is surjective.
\end{lemma}

\begin{proof}
Analogue of \cite[3.4.3]{BHS}. Twisting $\delta$ by a locally algebraic character, we may assume without loss of generality that $\wt_{\tau}(\delta)$ is either negative or does not belong to $\Z$, for any $\tau\in \Sigma$. By the hypothesis, we have $\dim_L H^1_{\hphigamma}(\rob_{L, K}(\delta))=\dim_L H^1_{\hphigamma}(\rob_{L, K}(\delta)\invertt)=[K:\Q_p]$. Let $s$ be the number of places $\tau\in \Sigma$ where $\wt_{\tau}(\delta)\in \Z$. Then we also have $\dim_L H^1(G_K, W_{\dR}(\rlk(\delta)\invertt))=s$. It thus suffices to show that the map
\[
H^1_{\hphigamma}(\rob_{L, K}(\delta)) \map H^1(G_K, W_{\dR}(\rlk(\delta)\invertt))
\]
has kernel of dimension at most $[K:\Q_p]-s$ over $L$. As in \cite[3.4.2]{BHS}, let $W(\delta):=(W_e(\rlk(\delta)), W^+_\dR(\rlk(\delta)))$ be the $L$-B-pair associated to $\rlk(\delta)$. The same argument as in \cite[3.4.2]{BHS}, using the duality theorem \cite[Proposition 2.11]{Nakamura} reduces the proof to showing that the map
\[\label{mapdual}
H^1(G_K, W(\delta^{-1}\epsilon)) \map H^1(G_K, W_e(\rlk(\delta^{-1}\epsilon)))
\]
has kernel of dimension at least $s$ over $L$. For this, we observe as in \cite[3.4.2]{BHS} that \ref{mapdual} factors through 
\[
H^1(G_K, W(\delta^{-1}\epsilon)) \map H^1(G_K, W(z^{-\mathbf{k}}\delta^{-1}\epsilon)) \cong H^1_{\hphigamma}(t^{-\mathbf{K}}\rlk(\delta^{-1}\epsilon))
\]
for any multi-index $\mathbf{k}\in \Z^{[K:\Q_p]}_{\geq 0}$. This map has kernel $H^1_{\hphigamma}(\rlk(z^{-\mathbf{k}}\delta^{-1}\epsilon)/t^{\mathbf{k}})$, which is of dimension precisely $s$ for any $\mathbf{k}$ large enough by Lemma \ref{modt}. Thus the map \ref{mapdual} has kernel of dimension $\geq s$ over $L$.
\end{proof}

We need the following variation of the notion $\mathcal{T}_0^n$ defined in \cite[3.4]{BHS} and a condition on the weights.
\begin{definition}\label{regular}
\begin{enumerate}
    \item We let $\mathcal{T}_1^n$ be the open analytic subspace of $\mathcal{T}^n$ consisting of character tuples $(\delta_1, \ldots, \delta_n)$ such that none of the $(\prod_{i\in S_1}\delta_i)\cdot(\prod_{j\in S_2}\delta_j)^{-1}$ or $\epsilon(\prod_{i\in S_1}\delta_i)\cdot(\prod_{j\in S_2}\delta_j)^{-1}$ are algebraic, for any subset $S_1, S_2\subset \{1, \ldots, n\}$ of the same cardinality and  $S_1\neq S_2$. 

    \item Recall from \cite[3.7]{BHS} $\mathcal{T}_\regu$ is the complement in $\mathcal{T}$ of the points $z^{-\mathbf{k}}$ and $\epsilon(z)z^{\mathbf{k}}$ for $\mathbf{k}\in \Z_{\geq 0}^\Sigma$. And let $\mathcal{T}^n_\regu$ be the Zariski open analytic subspace of $\mathcal{T}^n$ consisting of $(\delta_1, \ldots, \delta_n)$ such that $\delta_i/\delta_j\in \mathcal{T}_\regu$ for any $i\neq j$.

\item  We say an $n$-tuple of characters $(\delta_1, \ldots, \delta_n)\in \mathcal{T}^n(L)$ is regular if for any embedding $\tau: K\map L$, $\wt_\tau(\delta_1), \ldots, \wt_\tau(\delta_n)$ are all different.
 \end{enumerate}
\end{definition}

\begin{remark}
$\mathcal{T}_1\subset\mathcal{T}_0$. And for any     $(\delta_1', \ldots, \delta_n')$ such that $\delta_i'\delta_i^{-1}$ are algebraic for any $i$, $(\delta_1', \ldots, \delta_n')\in \mathcal{T}_1^n(L)$ if and only if $(\delta_1, \ldots, \delta_n)\in \mathcal{T}_1^n(L)$.
\end{remark}

\begin{lemma}\label{uniquetri}
Let $A\in \mathcal{C}_L$ and $\mathcal{M}$ be a trianguline $\phigamma$-module over $\rob_{A, K}\invertt$ with parameters $(\delta_{A, 1}, \ldots, \delta_{A, n})\in \mathcal{T}^n(A)$ such that their reductions $(\delta_1, \ldots, \delta_n)\in \mathcal{T}_1^n(L)$. Then $\mathcal{M}$ has a unique triangulation with parameter $(\delta_1, \ldots, \delta_n)$ and there is a unique  quotient of $\wedge^i \mathcal{M}$  that is isomorphic to $\rob_{A, K}(\prod_{j=n-i+1}^n\delta_j)\invertt$.
\end{lemma}

\begin{proof}
A devissage using the triangulation on $\mathcal{M}$ reduces the proposition to \ref{h0prop}.
\end{proof}

Let $\mathcal{M}$ be a trianguline $(\varphi, \Gamma_K)$-module of rank $n$ over $\rob_{L, K}[\frac{1}{t}]$, $\mathcal{M}_\bullet$ be a triangulation of $\mathcal{M}$ and $\underline{\delta}=(\delta_1, \ldots, \delta_n)$ be a parameter of $\mathcal{M}_\bullet$. For each $\tau\in \Sigma$, decompose $\{1, \ldots, n\}$ into classes of integral weight differences $S_{\tau, 1}, \ldots, S_{\tau, k_\tau}$ of cardinality $n_{\tau, 1}, \ldots, n_{\tau, k_\tau}$ as above. And we choose $\chi_1, \ldots, \chi_k$ accordingly. We invoke the notations defined before Corollary \ref{commdiag1} and let $W:=W_{\dR}(\mathcal{M})=\bigoplus_{\tau\in \Sigma}W_\tau$,  and $\mathcal{F}_\bullet=\bigoplus_{\tau\in \Sigma}\mathcal{F}_{\tau, \bullet}$. Write $X_{W, \mathcal{F}_\bullet}:=\prod_{\tau\in \Sigma} X_{W_\tau, \mathcal{F}_{\tau, \bullet}}$, $X_W:=\prod_{\tau\in \Sigma}X_{W_\tau}$ and $\mathfrak{t}:=\bigoplus_{\tau\in \Sigma}\bigoplus_{l=1}^{k_\tau}\mathfrak{t}_{\tau, l}$. By taking the product over $\tau\in \Sigma$ of the corresponding maps, we have the map 
\[
\kappa_{W, \mathcal{F}_\bullet}: X_{W, \mathcal{F}_\bullet}\map \widehat{\mathfrak{t}}
\]
where the right hand side is the completion of $\mathfrak{t}$ at $0$, and a map 
\[
\wt-\wt(\underline{\delta}): \widehat{\mathcal{T}^n_{\underline{\delta}}}\map \widehat{\mathfrak{t}}
\]
and the map
\[
X_{\mathcal{M}, \mathcal{M}_\bullet}\map  X_{W, \mathcal{F}_\bullet}
\].

\begin{proposition}\label{formsm1}
Notations as above. We have a similar commutative diagram as Corollary \ref{commdiag1} involving the above maps. The induced morphism
\[
X_{\mathcal{M}, \mathcal{M}_\bullet}\map \widehat{\mathcal{T}^n_{\underline{\delta}}}\times_{\widehat{\liet}} X_{W, \mathcal{F}_\bullet}
\]
of groupoids over $\mathcal{C}_L$ is formally smooth.
\end{proposition}

\begin{proof}
We will freely use the notations as in \cite[Theorem 3.4.4]{BHS}. The ingredients used in \cite[Theorem 3.4.4]{BHS} are
\begin{enumerate}
    \item The surjectivity of the map $H^1_{\hphigamma}(\mathcal{M}_{A, i-1}(\delta_{A, i}^{-1}))\map H^1(G_K, W_{\dR}(\mathcal{M}_{A, i-1}(\delta_{A, i}^{-1})))$.
    \item The isomorphism $H^1(G_K, W_{\dR}(\mathcal{M}_{A, i-1}(\delta_{A, i}^{-1})))\otimes_A B\cong H^1(G_K, W_{\dR}(\mathcal{M}_{B, i-1}(\delta_{B, i}^{-1})))$.
    \item The isomorphism $H^1_{\hphigamma}(\mathcal{M}_{A, i-1}(\delta_{A, i}^{-1}))\otimes_A B\cong H^1_{\hphigamma}(\mathcal{M}_{B, i-1}(\delta_{B, i}^{-1}))$.
\end{enumerate}

In our cases, (1) follows from a devissage argument using Lemma \ref{surjforsm}. For (2), we note that for any $W_{\tau}\in \mathrm{Rep}_{\pdr, A, \tau, \chi_1, \ldots, \chi_k}$, the module $H^1(G_K, W_\tau)$ can be computed as the  cokernel of $\nu_{W_\tau, l}$ on the $A$-module $D_{\pdr, \tau, l}(W_\tau)$ for $l$ the only index such that $\wt_\tau(\chi_l)\in \Z$ (the module interpretted as $0$ if no such $l$ exists). Thus the map
\[
H^1(G_K, W_\tau)\otimes_A B\map H^1(G_K, W_\tau\otimes_A B) \label{tensoreq}
\]
is either a trivial map of $0$ or the map induced by taking cokernel of $\nu_{W_\tau, l}\otimes_A B=\nu_{W_\tau\otimes_A B, l}$ on the isomorphism (by the last three lines of the proof of Lemma \ref{draeq})
\[
D_{\pdr, \tau, l}(W_\tau)\otimes_A B\cong D_{\pdr, \tau, l}(W_\tau\otimes_A B)
\]
and thus \ref{tensoreq} is an isomorphism. Let $W_\tau=W_{\dR, \tau}(\mathcal{M}_{A, i-1}(\delta_{A, i}^{-1}))$ and take the direct sum over $\tau\in \Sigma$ gives (2). (3) follows from precisely the same argument as in \cite[Theorem 3.4.4]{BHS}.

\end{proof}

\subsection{Local Model}

For a given $\phigamma$-module $D$ over $L$, let $\mathcal{M}:=D[\frac{1}{t}]$, $W^+:=W_{\dR}^+(D)$ and $W:= W_{\mathrm{dR}}(\mathcal{M})$. We have the commutative diagram

\[
\begin{tikzcd}
X_D \arrow[r] \arrow[d] & X_{W^+} \arrow[d]\\
X_{\mathcal{M}} \arrow[r] & X_{W}
\end{tikzcd}
\]

\begin{lemma}\label{bpair}
The morphism $X_D\map X_{\mathcal{M}} \times_{X_W}  X_{W^+}$ induced by the commutative diagram above is an equivalence.
\end{lemma}

\begin{proof}
Identical to \cite[3.5.1]{BHS}.   
\end{proof}

Set $X_{D, \mathcal{M}_\bullet}:=X_D\times_{X_{\mathcal{M}}}X_{\mathcal{M}, \mathcal{M}_\bullet}$ and $X_{W^+, \mathcal{F}}:=X_{W^+}\times_{X_W}X_{W, \mathcal{F}_\bullet}$ as in \cite[3.5]{BHS}. Let $r: G_K\map GL_n(L)$ be a continuous representation, $X_r$ be the groupoid of framed deformations of $r$ over $\mathcal{C}_L$, and $X_{r, \mathcal{M}_\bullet}:=X_r\times_{X_D}X_{D, \mathcal{M}_\bullet}$.  The following corollary follows from Lemma \ref{bpair} and Lemma \ref{formsm1} the same way \cite[3.5.6]{BHS} follows from \cite[3.5.3]{BHS} and \cite[3.4.4]{BHS}. There are corresponding local deformations with a framing denoted by a superscript $\square$. Note that here the framing is always on the $D_\pdr(\mathcal{M}_A)$ instead of the representation $r_A$.

\begin{cor}\label{smoothness}
The morphism $X_{D, \mathcal{M}_\bullet}\map X_{W^+, \mathcal{F}_\bullet}$ of groupoids over $\mathcal{C}_L$ is formally smooth. 
\end{cor}

Recall that there are two filtrations defined on each of the $D_{\pdr, \tau, l}(r)$, for any $\tau\in \Sigma$ and $\l\in \{1, \ldots, k_\tau\}$: The first one $\mathcal{D}_{\tau, l, \bullet}$ is  induced by the triangulation, introduced before Corollary \ref{xwf}. It satisfy the property that $\gr^i(\mathcal{D}_{\tau, l, \bullet}):=\mathcal{D}_{\tau, l, i}/\mathcal{D}_{\tau, l, i-1}$ is rank-$1$ free over $A$ if and only if $i\in S_{\tau, l}$ and is $0$ otherwise. 

The second filtration $\Fil_{W^+, \tau, l, \bullet}$ is induced by the de Rham filtration, defined before Corollary \ref{xwplus}, by reindexing the filtration defined in Definition \ref{def313}. The Hodge-Tate weights are given by the indices $\{\wt_\tau(\delta_i)_{i\in S_{\tau, l}}\}-\wt_\tau(\chi_{\tau, l})$. If $r$ is regular, the filtration gives a complete flag.


For each $\tau$ and $l$ as above, fix a trivialization $\alpha_{\tau, l}: L^{n_{\tau, l}}\cong D_{\mathrm{pdR}, \tau, l}(r)$. Then the triple $(\alpha_{\tau, l}^{-1}(\mathcal{D}_{\tau, l, \bullet}), \alpha_{\tau, l}^{-1}(\Fil_{W^+, \tau, l, \bullet}), N_{W, \tau, l})$ defines a $L$-point in $X_{\tau, l}:=\widetilde{\mathfrak{g}}_{n_{\tau, l}}\times_{\mathfrak{g}_{n_{\tau, l}}}\widetilde{\mathfrak{g}}_{n_{\tau, l}}$. Take product over $\tau$ and $l$, we get a point $x\in X(L):=\prod_{\tau\in \Sigma}\prod_{l=1}^{k_\tau}X_{\tau, l}(L)$. 

\begin{definition}\label{assoc}
For $r: G_K\map GL_n(L)$ a Galois representation and $\mathcal{M}_\bullet$ a triangulation on $\mathcal{M}=D_{\rig}(r)\invertt$, the above construction gives a point $x\in X(L)$. We say $x$ is the points in $X$ associated with the tuple $(r, \mathcal{M}_\bullet)$. And we let $w_x\in \mathcal{S}$ denote the relative position of the two flags given by $x$. $w_x$ does not depend on the trivialization.
\end{definition}

For any $w=(w_{\tau, l})_{\tau, l}\in \mathcal{S}$, we let $X_{\tau, l, w}$ be the irreducible component of $X_{\tau, l}$ labelled by $w_{\tau, l}$ and set $X_w:= \prod_{\tau, l}X_{\tau, l, w}$ be an irreducible component of $X$.  Taking product of the various maps $X_{W, \mathcal{F}_\bullet}^\square\map \widehat{\widetilde{\mathfrak{g}}}_{n_{\tau, l}}$ and $X_{W^+}^\square\map \widehat{\widetilde{\mathfrak{g}}}_{n_{\tau, l}}$ over $X_W^\square\map \mathfrak{g}_{n_{\tau, l}}$, we obtain a map $X_{W^+, \mathcal{F}_\bullet}^\square\map \widehat{X}_x$, and a natural composition map $X_{D, \mathcal{M}_\bullet}^\square\map X_{W^+, \mathcal{F}_\bullet}^\square\map \widehat{X}_x$. Furthermore, we define $\Theta$ as the composition map: 
\[X_{r, \mathcal{M}_{\bullet}}^\square \map X_{D, \mathcal{M}_\bullet}^\square\map \widehat{X}_x \map   \widehat{T}_{ (0, 0)}
\]
where the last term is defined as $\prod_{\tau\in \Sigma}\prod_{l=1}^{k_\tau} \widehat{T}_{\tau, l, (0, 0)}$, and the last map is defined by taking product over the completion of each $(\kappa_{1, \tau, l}, \kappa_{2, \tau, l}): X_{\tau, l}\map T_{\tau, l}$. The map factors through $X_{r, \mathcal{M}}$ and we denote the induced map $X_{r, \mathcal{M}_{\bullet}}\map \widehat{T}_{ (0, 0)}$ by $\Theta$ again by abuse of notation. The map $\mathrm{pr}_1\circ \Theta: X_{r, \mathcal{M}_{\bullet}}\map \widehat{\mathfrak{t}}$ factors through $X_{\mathcal{M}, \mathcal{M}_\bullet}$ and the map $\mathrm{pr}_2\circ \Theta: X_{r, \mathcal{M}_{\bullet}}\map \widehat{\mathfrak{t}}$ factors through $X_{W^+}$.

\begin{cor}\label{geometry}
\begin{enumerate}
    \item The groupoid $X_{W^+, \mathcal{F}_\bullet}^\square$ over $\mathcal{C}_L$ is pro-representable by the formal scheme  $\widehat{X}_x$ via the natural map defined above.
    \item The groupoid $ X_{D, \mathcal{M}_\bullet}^\square$ over $\mathcal{C}_L$   is pro-representable, by a formal scheme which is formally smooth over $\widehat{X}_x$. 
    \item The formal scheme representing $ X_{D, \mathcal{M}_\bullet}^\square$ has dimension $[K:\Q_p](n^2+\frac{n(n+1)}{2})$ and $\widehat{X}_x$ has dimension $\sum_{\tau\in \Sigma}\sum_{l=1}^{k_\tau}n_l^2$.
\end{enumerate}    
\end{cor}

\begin{proof}
Using Corollary \ref{smoothness}, Corollary \ref{xwf} and Corollary \ref{xwplus}, we see all the claim except the ones on dimension. There is a pullback diagram
\[
\begin{tikzcd}
X_{D, \mathcal{M}_\bullet}^\square \arrow[r] \arrow[d] & X_{\mathcal{M}, \mathcal{M}_\bullet}^\square \arrow[d]\\
X_{W^+, \mathcal{F}_\bullet}^\square \arrow[r] & X_{W, \mathcal{F}_\bullet}^\square
\end{tikzcd}
\]
where the column maps are formally smooth. Now \cite[3.5.7]{BHS} (it works with general trianguline $\mathcal{M}$ over $\rob_{L, K}\invertt$) gives that $X_{\mathcal{M}, \mathcal{M}_\bullet}^\square$ have dimension  $[K:\Q_p](n^2+\frac{n(n+1)}{2})$. Furthermore $X_{W^+, \mathcal{F}_\bullet}^\square$ has dimension $\sum_{\tau\in \Sigma}\sum_{l=1}^{k_\tau} \dim \widetilde{\mathfrak{g}}_{n_{\tau, l}}\times_{\mathfrak{g}_{n_{\tau, l}}}\widetilde{\mathfrak{g}}_{n_{\tau, l}}=\sum_{\tau\in \Sigma}\sum_{l=1}^{k_\tau}n_l^2$, and $X_{W, \mathcal{F}_\bullet}^\square \cong \prod_{\tau\in \Sigma}\prod_{l=1}^{k_\tau}\widetilde{\mathfrak{g}}_{n_{\tau, l}}$ also has dimension $\sum_{\tau\in \Sigma}\sum_{l=1}^{k_\tau}n_l^2$, we conclude that $\dim X_{D, \mathcal{M}_\bullet}^\square=\dim X_{\mathcal{M}, \mathcal{M}_\bullet}^\square=[K:\Q_p](n^2+\frac{n(n+1)}{2})$.
\end{proof}

\begin{definition}\label{sx}
\begin{enumerate}
    \item  For any $w\in \mathcal{S}$,  set $X_{W^+, \mathcal{F}_\bullet}^{\square, w}:=X_{W^+, \mathcal{F}_\bullet}^\square\times_{\widehat{X}_x} \widehat{X}_{w, x}$,  $ X_{D, \mathcal{M}_\bullet}^{\square, w}:=X_{D, \mathcal{M}_\bullet}^\square \times_{\widehat{X}_x} \widehat{X}_{w, x}$ and  $X_{r, \mathcal{M}_{\bullet}}^{\square, w}:= X_{r, \mathcal{M}_\bullet}^\square \times_{\widehat{X}_x} \widehat{X}_{w, x}$.
    \item Let $\mathcal{S}(x):=\{w\in \mathcal{S}: x\in X_w(L)\}=\{w\in \mathcal{S}: \widehat{X}_{w,x}\neq \emptyset \}$.
\end{enumerate}
\end{definition}

\begin{cor}\label{irredcomp}
\begin{enumerate}
    \item The irreducible components of $X_{W^+, \mathcal{F}_\bullet}^\square$ (resp. $X_{D, \mathcal{M}_\bullet}^\square$, $X_{r, \mathcal{M}_\bullet}^\square$) are given by $X_{W^+, \mathcal{F}_\bullet}^{\square, w}$ (resp. $X_{D, \mathcal{M}_\bullet}^{\square, w}$, $X_{r, \mathcal{M}_\bullet}^{\square, w}$), where $w\in \mathcal{S}(x)$. All the irreducible components are of the same dimension.
    \item The irreducible components of $X_{W^+, \mathcal{F}_\bullet}$ (resp. $X_{D, \mathcal{M}_\bullet}$, $X_{r, \mathcal{M}_\bullet}$) are given by $X_{W^+, \mathcal{F}_\bullet}^{ w}$ (resp. $X_{D, \mathcal{M}_\bullet}^{w}$, $X_{r, \mathcal{M}_\bullet}^{ w}$) characterized by the property that the pullback of each irreducible components of the unframed deformation rings labelled by $w$ is the irreducible component of the framed deformation rings of the same label.
\end{enumerate}
\end{cor}

\begin{proof}
(1) follows from the formally smooth property Corollary \ref{geometry} (1), (2), and that $\{\widehat{X}_{w, x}\}_{w\in \mathcal{S}(x)}$ are the irreducible components of $\widehat{X}_x$, of the same dimensions. (2) follows since the map from each framed deformation rings to unframed ones are formally smooth.
\end{proof}

\begin{proposition}\label{existmaxdim}
Fix $w\in \mathcal{S}$.  The formal scheme $ X_{r, \mathcal{M}_\bullet}^\square\cap \Theta^{-1}(\widehat{T}_{w, (0, 0)})$ contains an irreducible component of maximal dimension $[K:\Q_p](n^2+\frac{n(n+1)}{2})+n^2$ if and only if $w\in \mathcal{S}(x)$. Similar claim holds for $X_{r. \mathcal{M}_\bullet}\cap \Theta^{-1}(\widehat{T}_{w, (0, 0)})$ with the maximal dimension replaced by $[K:\Q_p]\frac{n(n+1)}{2}+n^2$.
\end{proposition}

\begin{proof}
The formal scheme representing $X_{r, \mathcal{M}_\bullet}^\square$ has dimension $[K:\Q_p](n^2+\frac{n(n+1)}{2})+n^2$ by Corollary \ref{geometry} (the extra $n^2$ coming from the framing of $r$). Its irreducible components are given by $X_{r, \mathcal{M}_{\bullet}}^{\square, w'}$, for $w'\in \mathcal{S}(x)$, all having dimension $[K:\Q_p](n^2+\frac{n(n+1)}{2})+n^2$, by Corollary \ref{irredcomp}. For those irreducible components, by Corollary \ref{irredcomp} again, $X_{r, \mathcal{M}_{\bullet}}^{\square, w'}\cap \Theta^{-1}(\widehat{T}_{w, (0, 0)})$ is a proper Zariski-closed subset of $X_{r, \mathcal{M}_{\bullet}}^{\square, w'}$ if and only if the map $\widehat{X}_{w',x}\map \widehat{T}_{(0, 0)}$ does not factor through $\widehat{T}_{w, (0, 0)}$, and in this case $X_{r, \mathcal{M}_{\bullet}}^{\square, w'}\cap \Theta^{-1}(\widehat{T}_{w, (0, 0)})$ have dimension $<[K:\Q_p](n^2+\frac{n(n+1)}{2})+n^2$. Now we see that to have an irreducible component of maximal dimension $[K:\Q_p](n^2+\frac{n(n+1)}{2})+n^2$ is equivalent to having a $w'\in \mathcal{S}(x)$ such that the map $\widehat{X}_{w',x}\map \widehat{T}_{(0, 0)}$ factors through $\widehat{T}_{w, (0, 0)}$, which is equivalent to $w'=w$ by \cite[2.5.2]{BHS}. We conclude that the existence of a maximal dimension component is equivalent to $w=w'\in \mathcal{S}(x)$.
\end{proof}

\begin{remark}
In fact by the proof, when $w\in \mathcal{S}(x)$, there is a unique irreducible component of $ X_{r, \mathcal{M}_\bullet}^\square\cap \Theta^{-1}(\widehat{T}_{w, (0, 0)})$ having the maximal dimension $[K:\Q_p](n^2+\frac{n(n+1)}{2})+n^2$.
\end{remark}

Next we prove an analogue to \cite[3.7.8]{BHS}. For this, we need to introduce a permutation group element $w$ associated to a point $z=(r, \delta_1, \ldots, \delta_n)\in X_{\tri}(L)$. By  \cite[Proposition 2.9]{modular} (or \cite[6.2.12]{KPX}), the set of $\tau$-Hodge-Tate-Sen weights of $r$ is the same as $\{\wt_\tau(\delta_1), \ldots, \wt_\tau(\delta_n)\}$, for any $\tau\in \Sigma$. Assume $r$ is regular, this is equivalent to $(\delta_1, \ldots, \delta_n)$ being regular. Fixing a $\tau\in \Sigma$, for each  $l\in \{1, \ldots, k_\tau\}$, we can thus find a cardinality-$n_l$ subset of the $\tau$-Hodge-Tate-Sen weights of $r$ that consists precisely of those  $\wt_\tau(\delta_i)$ for all $i\in S_{\tau, l}$. We may order those $\tau$-Hodge-Tate-Sen weights under the partial order introduced in the beginning of Section \ref{formdefprob} on elements within the same integral difference class, so that one write it as
$$
 \left( h_{\tau, l, 1} >\ldots >h_{\tau, l, n_l}\right)
$$

Let $a_1<\ldots < a_{n_{\tau, l}}$ be a listing of the elements in $S_{\tau, l}$ in their usual order. Now since the $\tau$-Hodge-Tate-Sen weights of $r$ are regular, for each $\tau$ and $l$, there exists a unique $w_{\tau, l}\in \mathcal{S}_{\tau, l}=\mathrm{Aut}(S_{\tau, l})$ such that 
\[
\left( \wt_\tau(\delta_{w_{\tau, l}^{-1}(a_1)}), \ldots, \wt_\tau(\delta_{w_{\tau, l}^{-1}(a_{n_{\tau, l}})}) \right)_\tau=\left( h_{\tau, l, 1}, \ldots ,h_{\tau, l, n_{\tau, l}}\right)_\tau
\]
. In other words, the permutation $w_{\tau, l}$ brings the $\tau$-Hodge-Tate-Sen weights of the characters $\delta_{a_1}, \ldots, \delta_{a_{n_{\tau, l}}}$ into an decreasing sequence.

\begin{definition}\label{xtriassoc}
Given $z\in X_\tri(L)$.   By the discussion above, we have an permutation element $w:=(w_{\tau, l})_{\tau, l}\in \mathcal{S}$. We say $w$ is the  permutation element in $\mathcal{S}$ associated to $z$.     
\end{definition}

\begin{remark}
    While $w$ is associated with $z$, the definition only depends on the characters $\delta_1, \ldots, \delta_n$ and not  on the Galois representation $r$ by the last sentence before Definition \ref{xtriassoc}.
\end{remark}

Note that by \cite[3.7.1]{BHS} or the proof of \cite[Theorem 6.3.13]{KPX} , there exists a unique triangulation $\mathcal{M}_\bullet$ on $\mathcal{M}$, with parameters $\delta_1, \ldots, \delta_n$. The same proof as in \cite[3.7.2]{BHS} yields that in our situation, there is a morphism $\widehat{X_\tri(\overline{r})}_z\map X_{r, \mathcal{M}_\bullet}$ over $X_r$. Again let $x$ be the points in $X$ associated with $r$ and its unique triangulation on $\mathcal{M}$. We have the following analogue to \cite[3.7.8]{BHS}.

\begin{proposition}\label{nece}
Let $w$ be the permutation element associated to $z\in X_\tri(\overline{r})(L)$. Then $\widehat{X_\tri(\overline{r})}_z\map X_{r, \mathcal{M}_\bullet}$ induces an isomorphism $\widehat{X_\tri(\overline{r})}_z\map X_{r, \mathcal{M}_\bullet}^w$. In particular, $w\in \mathcal{S}(x)$.
\end{proposition}

\begin{proof}
Same proof as \cite[3.7.3]{BHS} shows that the map $\widehat{X_\tri(\overline{r})}_z\map X_{r, \mathcal{M}_\bullet}$ is a closed immersion.  Let $\Theta_x$ be the composite map 
\[
\begin{tikzcd}
    \widehat{X_\tri(\overline{r})}_x\arrow[r] &  X_{r, \mathcal{M}_\bullet} \arrow[r, "\Theta"] & \widehat{T}_{(0, 0)}
\end{tikzcd}
\]
We claim $\Theta_x$ factors through $\widehat{T}_{w, (0, 0)}\inj \widehat{T}_{(0, 0)}$.  We start with a point  $\widetilde{z}=(r_A, \delta_{A, 1}, \ldots, \delta_{A, n})\in X_\tri(\overline{r})(A)$. Now it is striaghtforward that $\mathrm{pr}_1( \Theta_x(\widetilde{z}))=(\wt_\tau(\delta_{A, i})-\wt_\tau(\delta_i))_{i, \tau}$. We claim  $\mathrm{pr}_2( \Theta_x(\widetilde{z}))=(\nu_{\tau, l, 1}, \ldots, \nu_{\tau, l, n_{\tau, l}})_{\tau, l}$ satisfy the property that 
\[
f(Y)= \prod_{l=1}^{k_\tau}\prod_{i=1}^{n_{\tau, l}}(Y-h_{\tau, l, i}-\nu_{\tau, l, i}) \label{senpoly}
\]
is the $\tau$-Hodge-Tate-Sen polynomial of $r_A$ (See the paragraphs before Definition \ref{xtriassoc} for notation). In fact, let $W_{A, \tau}^+:=r\otimes_{\tau, K} B^+_{\dR}$, then by the proof of  Lemma \ref{drplusaeq}, one can decompose $W_{A, \tau}^+=\bigoplus_{l=1}^{k_\tau}  W_{A, \tau, l}^+$, where each $W_{A, \tau, l}^+\in \mathrm{Rep}^+_{\pdr, A, \tau, \chi_l}$. Here we let $\chi_l$ be any character $\delta_i$ with $i\in S_{\tau, l}$. Thus $W_{A, \tau}^+(\chi_l^{-1})$ is almost de Rham, and by \cite[3.7.5]{BHS},  $W_{A, \tau}^+(\chi_l^{-1})/t_\tau$ can be written as a direct sum of rank $1$ free modules over $A\otimes_{\tau, K}C$ where the Sen operator acts by $h_{\tau, l, i}-\wt_\tau(\chi_l)+\nu_{\tau, l, i}$. Twisting back, we see the Sen polynomial of $W_{A, \tau}^+(\chi_l^{-1})/t$ is \ref{senpoly}. We know by \cite[3.7.6]{BHS} that for each $\tau$, 
\[
\{h_{\tau, l, j}+\nu_{\tau, l, j}\}_{l, i}=\{\wt_\tau(\delta_{A, i})\}_{l, i}
\]
as sets. After moding out the maximal ideal of $A$, one obtain 
\[
\{h_{\tau, l, j}\}_{l, i}=\{\wt_\tau(\delta_{i})\}_{l, i}
\] 
And it is immediate that both equality of sets holds with in the same integer difference class, i.e. for each fixed $l$. Now that $w_{\tau, l}$ brings  $\left(\wt_\tau(\delta_{a_1}), \ldots, \wt_\tau(\delta_{a_{n_{\tau, l}}})\right)$ into an decreasing order $h_{\tau, l,1}>\cdots > h_{\tau, l, n_{\tau, l}}$, the same $w_{\tau, l}$ brings \[\left(\wt_\tau(\delta_{A, a_1}), \ldots, \wt_\tau(\delta_{A, a_{n_{\tau, l}}})\right)\]
to \[\left(h_{\tau, l,1}+\nu_{\tau, l, 1}, \ldots,  h_{\tau, l, n_{\tau, l}}+\nu_{\tau, l, n_{\tau, j}}\right)\] 
And thus $w$ brings \[\left(\wt_\tau(\delta_{A, a_1})-\wt_\tau(\delta_{a_1}), \ldots, \wt_\tau(\delta_{A, a_{n_{\tau, l}}})-\wt_\tau(\delta_{a_{n_{\tau, l}}})\right)\]
to \[\left(\nu_{\tau, l, 1}, \ldots, \nu_{\tau, l, n_{\tau, j}}\right)\], proving our claim.

The rest follows the same way as in \cite[3.7.8]{BHS}: By a comparing dimension we see that $\widehat{X_\tri(\overline{r})}_x$ is a union of irreducible components $X_{r, \mathcal{M}_\bullet}^{w'}$ of $X_{r, \mathcal{M}_\bullet}$ under the embedding. The only $w'$ such that $\Theta: X_{r, \mathcal{M}_\bullet}^{w'}\map \widehat{T}_{(0, 0)}$ factors through $\widehat{T}_{w,(0, 0)}$ is $X_{r, \mathcal{M}_\bullet}^{w}$. This shows the first claim. Since $\widehat{X_\tri(\overline{r})}_x$ is nonempty, we immediately see by Definition \ref{sx} that $w\in \mathcal{S}(x)$.
\end{proof}

\section{The proof}\label{laproof}

From now on, we fix a point $z=(r, \delta_1, \ldots, \delta_n)\in M_c(L)$ where $(\delta_1, \ldots, \delta_n)\in \mathcal{T}_1^n(L)$ and is regular (Definition \ref{regular}). For each $\tau\in \Sigma$, we decomposed $\{1, \ldots, n\}$ into classes of integral weight differences $S_{\tau, 1}, \ldots, S_{\tau, k_\tau}$ of cardinality $n_{\tau, 1}, \ldots, n_{\tau, k_\tau}$ as in Section \ref{formdefprob}. We set $D:=D_{\rig}(r)$, $\mathcal{M}:=D_{\rig}(r)\invertt$ and $W$, $\mathcal{F}_\bullet$ etc. to be defined as in the paragraphs above Proposition 3.24 and Proposition 3.25 according to our $D$ and $\mathcal{M}$. Also we fix a trivialization $\alpha_{\tau, l}: L^{n_{\tau, l}}\cong D_{\mathrm{pdR}, \tau, l}(r)$ for each $\tau$ and $l$ as above.

We let $w_{z, \tau, l}\in \mathcal{S}_{\tau, l} $ be the relative position of the two flags $\alpha_{\tau, l}^{-1}(\mathcal{D}_{\tau, l, \bullet})$ and $\alpha_{\tau, l}^{-1}(\Fil_{W^+, \tau, l, \bullet})$ on $L^{n_{\tau, l}}$. Let $w_z:=(w_{z, \tau, l})_{\tau, l}\in \mathcal{S}$. The element $w_{z}$ does not depend on the choice of any $\alpha_{\tau, l}$.


Let $\widehat{M}_{c, x}$ be the completion of $M_c$ at its $L$-point $x$, viewed as a groupoid (or a deformation problem) over $\mathcal{C}_L$.  We have a natural map 
\[
p: \widehat{M}_{c, z} \map X_{r, \mathcal{M}_{\bullet}}
\]
defined by sending $(r_A, \delta_{A, 1 }, \ldots, \delta_{A, n})$ to $r_A$ with the filtration on $D_\rig (r_A)[\frac{1}{t}]$ induced by maps $f_i[\frac{1}{t}]$ coming from the nonzero maps $f_i$ spanning 
\[
\homo_{\varphi, \Gamma_K}(D_\rig(\bigwedge^{n-i+1} r_A), t^{-c}\rob_{A, K} (\prod_{j=i}^n \delta_{A, j})) 
\]

\begin{lemma}\label{wtsame}
For any point $(r_A, \delta_{A, 1 }, \ldots, \delta_{A, n})\in \widehat{M}_{c, z}(A)$,  let $f(Y)$ be the  $\tau$-Sen polynomial of $r_A$, for any fixed $\tau\in \Sigma$. Then we have
\[
f(Y)=\prod_{i=1}^n (Y-\wt_\tau(\delta_{A, i}))
\]
. In particular, if $(r, \delta_{1 }, \ldots, \delta_{n})\in M_{c}(L)$ and the representation $r$ has $\tau$-Hodge-Tate-Sen weights $\{h_{\tau, 1}, \ldots, h_{\tau, n}\}$ then $\{h_{\tau, 1}, \ldots, h_{\tau, n}\}=\{\wt_\tau(\delta_{ 1}), \ldots, \wt_\tau(\delta_{n})\}$.
\end{lemma}

\begin{proof}
By inductively using Lemma \ref{htwt}, $D_{\mathrm{Sen}, \tau}(r_A)$ has a filtration whose graded pieces are rank-$1$ free $A\otimes_{\tau, K}K_\infty$ modules where the Sen operator acts by the scalar $\wt_\tau(\delta_{A, i})$ for $i\in \{1, \ldots, n\}$. This immediately implies the form of the  $\tau$-Sen polynomial as stated.
\end{proof}

Given $z=(r, \delta_1, \ldots, \delta_n)\in M_c(L)$.  Note that by Lemma \ref{wtsame}, we are again in the situation that the the set of $\tau$-Hodge-Tate-Sen weights of $r$ is the same as $\{\wt_\tau(\delta_1), \ldots, \wt_\tau(\delta_n)\}$. So again $r$ is regular if and only if $(\delta_1, \ldots, \delta_n)$ is regular.  A similar procedure as in the paragraphs preceding Definition \ref{xtriassoc} produce an element $w:=(w_{\tau, l})_{\tau, l}\in \mathcal{S}$, that brings the the $\tau$-Hodge-Tate-Sen weights of the charaters within the same integral difference class to decreasing order.

\begin{definition}\label{assocwt}
Given $z\in M_c(L)$.   The $w$ obtained above is  said to be the  permutation element in $\mathcal{S}$ associated to $z$. 
\end{definition}

\begin{lemma}\label{facthrough}
The map $p$ factor through the inclusion 
\[
X_{r, \mathcal{M}_{\bullet}} \cap \Theta^{-1}(\widehat{T}_{w, (0, 0)}) \inj X_{r, \mathcal{M}_{\bullet}}
\]
\end{lemma}

\begin{proof}

For any $A\in \mathcal{C}_L$, a point $\widetilde{z}=(r_A, \delta_{A, 1 }, \ldots, \delta_{A, n})\in \widehat{M}_{c, z}(A)$ and an $l\in \{1, \ldots, k\}$, consider the two complete flags $\mathcal{D}_{A, \tau, l, \bullet}$ and $\Fil_{W^+_A, \tau, l, \bullet}$ on $D_{\pdr, \tau, l}(r_A)$. Let $(\nu_{\tau, l, 1}, \ldots ,\nu_{\tau, l, n_{\tau, l}})_{\tau, l}$ be the image of $\widetilde{z}$ under the composite map
\[
\begin{tikzcd}
\widehat{M}_{c, z} \arrow[r, "p"] & X_{r, \mathcal{M}_\bullet}\arrow[r, "\mathrm{pr}_2\circ\Theta"] &  \widehat{\mathfrak{t}} 
\end{tikzcd}
\]
by definition we see that the nilpotent operator $N_{W_A, \tau, l}$ action on the graded pieces of $\Fil_{W_A^+, \tau, l, \bullet}$, ordered in increasing order of indices, are given by $(\nu_{\tau, l, 1}, \ldots ,\nu_{\tau, l, n_{\tau, l}})$ for each $\tau\in \Sigma$ and $l\in \{1, \ldots, k_\tau\}$. Thus, by applying a twisted version of \cite[3.7.5]{BHS} to each $W_{A, l}^+$ (In the notation of Lemma 3.15), as in the proof of Proposition \ref{nece}, we see the $\tau$-Hodge-Tate-Sen polynomial of $r_A$ is given by 

\[
f(Y)= \prod_{l=1}^{k_\tau}\prod_{i=1}^{n_{\tau, l}}(Y-h_{\tau, l, i}-\nu_{\tau, l, i})
\]

On the other hand, each graded piece of $\mathcal{D}_{A, \tau, l, \bullet}$ is of the form $D_\pdr(\rob_{L, K}(\delta_{A, a_s})[\frac{1}{t}])$, for some $a_s\in S_{\tau, l}$. If $\mathrm{pr}_1\circ\Theta\circ p(\widetilde{z})=(\mu_{\tau, l, a_1}, \ldots ,\mu_{\tau, l, a_{n_{\tau, l}}})_{\tau, l}$, then by definition $\wt_\tau(\delta_{A, a_i})=\wt_\tau(\delta_{a_i})+\mu_{\tau, l, a_i}$ for any $i\in \{1, \ldots, n_{\tau, l}\}$. Thus by Lemma \ref{wtsame}, we see the $\tau$-Hodge-Tate-Sen polynomial of $r_A$ is given by

\[
f(Y)=\prod_{l=1}^{k_\tau}\prod_{i=1}^{n_{\tau, l}}(Y-\wt_\tau(\delta_{ a_i}))
\]

Since for each $\tau$, the $\tau$-Hodge-Tate-Sen weights only depends on $r_A$ and not on the filtration, we see for each $\tau$, 
\[
\{h_{\tau, l, i}+\nu_{\tau, l, i}\}_{l, i}=\{\wt_\tau(\delta_{a_i})+\mu_{\tau, l, a_i}\}_{l, i}
\]
as sets. (Note that all $\wt_\tau(\delta_{a_i})$ are different).



By an argument similar to the proof of Proposition \ref{nece}, we see that the same $w_{\tau, l}$ brings $(\mu_{\tau, l, a_1}, \ldots, \mu_{\tau, l, a_{n_{\tau, l}}} )$ to $(\nu_{\tau, l, 1}, \ldots, \nu_{\tau, l, n_{\tau, l}})$. Thus proving the claim that 
\[ w \left(\mathrm{pr}_1\circ\Theta\circ p(\widetilde{z})\right) = \mathrm{pr}_2\circ\Theta\circ p(\widetilde{z})
\]
\end{proof}

\begin{remark}
The difference of the above proof with the proof of Proposition \ref{nece} is that we relate Hodge-Tate weights of the characters to Hodge-Tate weights of $r$ by the definition of the moduli problem $M_c$, while Proposition \ref{nece} relates them from the fact that $X_\tri(\overline{r})$ is the closure of a set of points that satisfy the relation on the weights.
\end{remark}

The following two propositions are the key steps to showing that there exists an irreducible component of $\widehat{M}_{c, z}$ with an abundant amount of points coming from regular trianguline points.
\begin{proposition}\label{descpmc}
Let $w\in \mathcal{S}$ be given by $z$ as before. The map 
\[
p: \widehat{M}_{c, z}\map X_{r, \mathcal{M}_{\bullet}} \cap \Theta^{-1}(\widehat{T}_{w, (0, 0)})
\]
is an isomorphism if $c> \sum_{j\in S_1}h_{\tau, j}-\sum_{j\in S_2}h_{\tau, j}$ for any $\tau\in \Sigma$ and $S_1, S_2\subset\{1, \ldots, n\}$ of the same cardinality such that $\sum_{j\in S_1}h_{\tau, j}-\sum_{j\in S_2}h_{\tau, j}\in \Z$.
\end{proposition}

\begin{proof}
We construct an inverse map $q$. Let $\widetilde{y}=(r_A, \mathcal{M}_{A, \bullet})$ be an $A$-point of the right hand side. There exists unique characters $\delta_{A, 1}, \ldots, \delta_{A, n}$ lifting $\delta_1, \ldots, \delta_n$ such that the  triangulation $\mathcal{M}_{A, \bullet}$  has parameters $\delta_{A, 1}, \ldots, \delta_{A, n}$ (by \cite[3.3.4]{BHS}). We set $q(\widetilde{y}):=(r_A, \delta_{A, 1}, \ldots, \delta_{A,n})\in \left(\spf (R^{\square}_{\overline{r}})^{\ad}_\eta \times \mathcal{T}^n\right)(A)$. Next we verify $q(\widetilde{y})\in M_c(A)$. By the condition that $(\delta_1, \ldots, \delta_n)\in \mathcal{T}_1^n(L)$, we apply Lemma \ref{uniquetri}  to see
\[
\homo_{\varphi, \gamma_K}(D_\rig(\bigwedge^{n-i+1} r_A)\invertt, \rob_{A, K} (\prod_{j=i}^n \delta_{A, j})\invertt) 
\]
is a free $A$-module of rank $1$. Moreover, we have a left exact sequence
\[
0 \map  H^0_{\hphigamma}(t^{-c}D_\rig(\bigwedge^{n-i+1} r_A)^\vee (\prod_{j=i}^n \delta_{A, j}))\map H^0_{\hphigamma}(D_\rig(\bigwedge^{n-i+1} r_A)^\vee (\prod_{j=i}^n \delta_{A, j})\invertt) \map
\]
\[
H^0_{\hphigamma}\left(\left(\rob_{A, K}\invertt/t^{-c}\rob_{A, K}\right)\otimes_{\rob_{A, K}}D_\rig(\bigwedge^{n-i+1} r_A)^\vee (\prod_{j=i}^n \delta_{A, j})\right) 
\]
We claim the last term is $0$. By a devissage argument, it suffices to show that $H^0_{\hphigamma}(\rob_{A, K}\invertt/t^{-c}\rob_{A, K}\otimes_{\rob_{A, K}}\rob_{A, K} (\prod_{l\in S}\delta_{A, l}^{-1}\prod_{j=i}^n \delta_{A, j}))=0$ for any $S\subset \{1, \ldots, n\}$ of cardinality $n-i+1$. Since $-c-\sum_{l\in S}\wt_\tau(\delta_l)+\sum_{j=i}^n \wt_\tau(\delta_j)$ is either not integer or $<0$, for any $\tau\in \Sigma$ and $S$ as above by the condition on $c$, we see the vanishing result from Lemma \ref{h0prop} (2). Thus 
\[\homo_{\varphi, \gamma_K}(D_\rig(\bigwedge^{n-i+1} r_A), t^{-c}\rob_{A, K} (\prod_{j=i}^n \delta_{A, j})) \cong \homo_{\varphi, \gamma_K}(D_\rig(\bigwedge^{n-i+1} r_A), \rob_{A, K} (\prod_{j=i}^n \delta_{A, j})\invertt) )
\]
is a free $A$-module of rank $1$.  Condition (1) and (3) of Definition \ref{moduliprob} are verified.

Now we verify condition (2) and (4) of Definition \ref{moduliprob}. We write
\[
\mathrm{pr}_1\circ\Theta( \widetilde{y})=(\mu_{\tau, l, a_1}, \ldots ,\mu_{\tau, l, a_{n_{\tau, l}}})_{\tau, l}
\]
and 
\[
\mathrm{pr}_2\circ\Theta( \widetilde{y})=(\nu_{\tau, l, 1}, \ldots ,\nu_{\tau, l, n_{\tau, l}})_{\tau, l}
\]
Just like in the proof of Lemma \ref{facthrough}, we immediately see from definition that  $\wt_\tau(\delta_{A, a_i})=\wt_\tau(\delta_{a_i})+\mu_{\tau, l, a_i}$ for any $i\in \{1, \ldots, n_{\tau, l}\}$, and the $\tau$-Hodge-Tate-Sen polynomial of $r_A$ is equal to
\[
f(Y)= \prod_{l=1}^{k_\tau}\prod_{i=1}^{n_{\tau, l}}(Y-h_{\tau, l, i}-\nu_{\tau, l, i})
\]
. The condition that $\widetilde{y}\in \Theta^{-1}(\widehat{T}_{w, (0, 0)})$ gives 
\[w_{\tau, l}^{-1}(\mu_{\tau, l, a_1}, \ldots, \mu_{\tau, l, a_{n_{\tau, l}}} )=(\nu_{\tau, l, 1}, \ldots, \nu_{\tau, l, n_{\tau, l}})
\]
. But the same $w_{\tau, l}$ also satisfy 
\[w_{\tau, l}^{-1}(\wt_\tau(\delta_{a_1}), \ldots, \wt_\tau(\delta_{a_{n_{\tau, l}}}))=(h_{\tau, l, 1}, \ldots, h_{\tau, l,   n_{\tau, l}})
\]
from the definition of $w$. Adding up the above two equation gives 
\[
\{\wt_\tau(\delta_{A, a_i})\}_{i=1}^{n_{\tau, l}}=\{h_{\tau, l, i}+\nu_{\tau, l, i}\}_{i=1}^{n_{\tau, l}}
\]
Thus, $D_{\mathrm{pHT}, \tau, l}(r_A)$ has a decomposition by finite free  rank-$1$ $A$-modules, where the Sen operators acts by precisely   $\wt_\tau(\delta_{A, a_1})-\wt_\tau(\chi_l), \ldots, \wt_\tau(\delta_{A, a_{n_{\tau, l}}})-\wt_\tau(\chi_l)$. Twisting back for each $l$, we see from the shape of $D_{\mathrm{pHT}, \tau, l}(r_A)$ that $D_{\mathrm{Sen}, \tau}(r_A)$ admits a filtration whose graded pieces are rank-$1$ free $A\otimes_{\tau, K}K_\infty$ modules where the Sen operator acts by the scalar $\wt_\tau(\delta_{A, i})$ for $i\in \{1, \ldots, n\}$. Thus by Lemma \ref{htwt}, condition (2) and (4) of Definition \ref{moduliprob} is verified for the point $q(\widetilde{y})$.

It is straightforward to check that the maps $p$ and $q$ are mutual inverse to each other. Thus concludes the proof.
\end{proof}

\begin{definition}
For each $w=(w_\tau)_{\tau\in \Sigma}\in \mathcal{S}$ and any $L'/L$, we set $X_{\tri, \mathrm{sp}}^{\mathrm{reg}, w}(L')$ to be the set of points
$(r, \delta_1, \ldots, \delta_n)\in U^{\mathrm{reg}}_{\tri}(L')$, such that $\wt_{\tau}(\delta_{w_{\tau}(i)})-\wt_{\tau}(\delta_i)\in \Z$ for any $\tau$ and $i\in \{1, \ldots, n\}$. We will show in the next Lemma that for each $w\neq 1$,  $X_{\tri, \mathrm{sp}}^{\mathrm{reg}, w}(L')$ is the $L'$ points of a countable collection of  Zariski-locally-closed rigid analytic subspace of  $\spf (R^{\square}_{\overline{r}})^{\ad}_\eta \times \mathcal{T}^n_{\mathrm{reg}}$ of smaller dimension than $[K:\Q_p]\frac{n(n+1)}{2}+n^2$. Thus defining $X_{\tri, \mathrm{sp}}^{\mathrm{reg}, w}$ as a countable union of rigid analytic space. For $w=1$, one simply recovers the usual $X_{\tri, \mathrm{sp}}^{\mathrm{reg}, 1}=U_{\tri}^{\mathrm{reg}}$ as in \cite[3.7]{BHS}, which is of dimension  $[K:\Q_p]\frac{n(n+1)}{2}+n^2$.

\end{definition}

\begin{lemma}\label{smallerd}
Let $w\in \mathcal{S}$ be a nontrivial element. Then $X_{\tri, \mathrm{sp}}^{\mathrm{reg}, w}$ is a countable union of (irreducible) Zariski locally closed rigid analytic subspace of $\spf (R^{\square}_{\overline{r}})^{\ad}_\eta \times \mathcal{T}^n_{\mathrm{reg}}$ of  dimension smaller than $[K:\Q_p]\frac{n(n+1)}{2}+n^2$.
\end{lemma}

\begin{proof}
Let $p_2$ be the projection $U_{\tri}^{\mathrm{reg}}\map \mathcal{T}^n_{\mathrm{reg}}$. Let $J$ be the set of pairs $(\tau, i)$, with $\tau\in \Sigma$ and $i\in \{1, \ldots, n\}$ such that $w_\tau(i)\neq i$. We consider the set $I$ given by the $\# J$-fold product of the set of  characters $K^\times \map (L')^\times $ over $J$, whose elements we write as $(\chi_{\tau, i})_{J}$. Let $I^{\mathrm{alg}}:=\{(\chi_{\tau, i})_{J}| \wt_\tau(\chi_{\tau, i})\in \Z, \forall (\tau, i)\in J\}$. By Lemma \ref{smchardim} below, this set is the $L'$ points of a countable union of Zariski-closed subanalytic space $\cup_{j\in \N}T_j$ of $\mathcal{T}_{\mathrm{reg}}^J$, of smaller dimension. In fact, it is the product over each factor $\mathcal{T}_{\mathrm{reg}}$ of a such union. We define a map of rigid analytic space $s_w: \mathcal{T}_\regu^n\map \mathcal{T}_\regu^J$ by the formula
\[
(\delta_i)_{i=1}^n \mapsto (\delta_{w_\tau(i)}\delta_i^{-1})_{(i, \tau)\in J}
\]
Now $X_{\tri, \mathrm{sp}}^{\mathrm{reg}, w}=p_2^{-1}s_w^{-1}(\cup_{j\in \N}T_j)$
. Hence $X_{\tri, \mathrm{sp}}^{\mathrm{reg}, w}$ is a countable union of Zariski-closde subanalytic spaces of $U_\tri^\regu$. 

To see the claim on dimension, we need only consider the projection of the map $s_w\circ p_2$ to one of the $\mathcal{T}_\regu$ corresponding to a $(\tau, i)\in J$. Composing $s_w$ with this projection gives a surjective homomorphism $p_w: \mathcal{T}_\regu^n\map \mathcal{T}_\regu$, which decomposes the group rigid analytic space $\mathcal{T}_\regu^n\cong \mathcal{T}_\regu^{n-1}\times \mathcal{T}_\regu$. The preimage under $p_w$ of any Zariski-closed rigid analytic subspace of dimension smaller than $\dim \mathcal{T}_\regu$ is thus of dimension $< \dim \mathcal{T}_\regu^n=n([K:\Q_p]+1)$. Now an argument similar to [20, 3,3] shows that the preimage under $p_2$ of any Zariski-closed rigid analytic subspace $T\subset \mathcal{T}_\regu^n$ is a successive vector bundle over $T$, which is of dimension $\dim T+[K:\Q_p]\frac{n(n-1)}{2}+n^2-n<[K:\Q_p]\frac{n(n+1)}{2}+n^2$.
\end{proof}

\begin{lemma}\label{smchardim}
Fix a $\tau\in \Sigma$. The set of continuous characters $\chi: K^\times\map (L')^\times$ in $\mathcal{T}_\regu$  whose $\tau$-weight is in $\Z$ is given by the $L'$-points of a countable union of Zariski-closed subanalytic space of $\mathcal{T}_\regu$ of dimension $<[K:\Q_p]+1$.
\end{lemma}

\begin{proof}
Immediately reduce to the weight space (with the dimension reduce by $1$). Furthermore, twisiting over all characters $x_\tau^n$, with $n$ ranging in $\Z$, reduce us to prove the claim for the set of characters of $\Of_K^\times$ with $\tau$-Hodge-Tate-Sen weight $=0$. Now the space of continuous characters of $\Of_K^\times$ is finite over the space of continuous characters of $1+p^n\Of_K$, for some fixed $n>2$ and this reduce us to the space of continuous characters of $1+p^n\Of_K$, which is represented by 
\[
\spf(\Of_L[[p^n\Of_K]])^{\ad}_\eta
\]
This is simply the $[K:\Q_p]$-fold product of open unit disc. There exists a system of coordinates $\{T_\tau\}_{\tau\in \Sigma}$, such that taking $\tau$-Hodge-Tate-Sen weight of a character $\chi$ is equivalent to evaluating the element $\log (1+T_\tau)$ on the $L'$-point corresponding to $\chi$. Thus the set of characters with $\tau$-Hodge-Tate-Sen weight $=0$ is given by the vanishing a locus of $\log (1+T_\tau)$, a nonzero elements over any of the closed polydisc contained in the $[K:\Q_p]$-fold product of open unit disc.
\end{proof}

We set $i_w: X_{\tri, \mathrm{sp}}^{\mathrm{reg}, w}\map \spf (R^{\square}_{\overline{r}})^{\ad}_\eta \times \mathcal{T}^n_{\mathrm{reg}}$ be the map defined by 
\[
(r, \delta_1, \ldots, \delta_n)\mapsto (r, \delta_1\prod_{\tau\in \Sigma}x_\tau^{\wt_\tau(\delta_{w_\tau(1)})-\wt_\tau(\delta_1)}, \ldots,  \delta_n\prod_{\tau\in \Sigma}x_\tau^{\wt_\tau(\delta_{w_\tau(n)})-\wt_\tau(\delta_n)} )
\]

\begin{lemma}\label{closedembed}
$X_{\tri, \mathrm{sp}}^{\mathrm{reg}, w}$ can be written as a countable union of  Zariski locally-closed subanalytic spaces of $\spf (R^{\square}_{\overline{r}})^{\ad}_\eta \times \mathcal{T}^n_{\mathrm{reg}}$, such that over each of the member the map $i_w$ is a locally closed embedding  of rigid analytic spaces into $\spf (R^{\square}_{\overline{r}})^{\ad}_\eta \times \mathcal{T}^n_{\mathrm{reg}}$.
\end{lemma}

\begin{proof}
We use the notation as in the proof of Lemma \ref{smallerd}. Invoking Lemma \ref{smchardim}, we see that then countable union $\cup_{j\in \N}T_j$ can be made such that over each $T_j$, all its point $(\chi_{\tau, i})_J$ satisfy $\wt_\tau(\chi_{\tau, i})=a_{\tau, i}$ for some fixed tuple of integers $a_{\tau, i}\in \Z^J$. Then $X_{\tri, \mathrm{sp}}^{\mathrm{reg}, w}$ can be written as a countable union of Zariski-closed  $\cup_{j\in \N} p_2^{-1}s_w^{-1}(T_j)$ such that for any points $(r, \delta_1, \ldots, \delta_n)$ lying in the same member of the union, we have $\wt_{\tau}(\delta_{w_\tau (i)})-\wt_\tau(\delta_i)=a_{\tau, i}$ is fixed. Thus, over each $p_2^{-1}s_w^{-1}(T_j)$, the map $i_w$ is the same as the restriction of an isomorphism of the rigid analytic spaces $\spf (R^{\square}_{\overline{r}})^{\ad}_\eta \times \mathcal{T}^n_{\mathrm{reg}}$ given by the formula
\[
(r, \delta_1, \ldots, \delta_n)\mapsto (r, \delta_1\prod_{\tau\in \Sigma}x_\tau^{a_{\tau, 1}}, \ldots, \delta_n\prod_{\tau\in \Sigma}x_\tau^{a_{\tau, n}})
\]
since all $a_{\tau, i}$ are constants. Because each $p_2^{-1}s_w^{-1}(T_j)$ are themselves Zariski-locally-closed in $\spf (R^{\square}_{\overline{r}})^{\ad}_\eta \times \mathcal{T}^n_{\mathrm{reg}}$, composing with the above isomorphism gives that the images $i_w(p_2^{-1}s_w^{-1}(T_j))$ are again Zariski-locally-closed in $\spf (R^{\square}_{\overline{r}})^{\ad}_\eta \times \mathcal{T}^n_{\mathrm{reg}}$.
\end{proof}

\begin{definition}
Let $M_c^\regu$ be the intersection of $M_c$ and the inverse image of $\mathcal{T}^n_\regu$.
\end{definition}

\begin{proposition}\label{covering}
$M_c^{\mathrm{reg}}$ is covered by the union of $i_w(X_{\tri, \mathrm{sp}}^{\mathrm{reg}, w})$, ranging over $w\in \mathcal{S}$. 
\end{proposition}

\begin{proof}
We show this on points. For any $z=(r, \delta_1, \ldots, \delta_n)\in M_c^{\mathrm{reg}}(L)$, the definition induces a complete flag $\mathcal{M}_\bullet$ of $\phigamma$-module  over $\rob_{L, K}\invertt$ on $\mathcal{M}:=D_{\rig}(r)\invertt$. Taking intersection with $D:=D_{\rig}(r)$, we see that there exists a filtration $\Fil_\bullet$ on $D$, whose graded pieces are rank-$1$ $\phigamma$-module over $\rob_{L, K}$, by the Bezout property of $\rob_{L, K}$. Since $\gr^i \mathcal{M}=\rob_{L, K}(\delta_i)\invertt$, we see that each $\gr^i D$ must be of the form $\rob_{L, K}(\delta_i')$, where $\delta_i'=\delta_i\prod_{\tau\in \Sigma}x_\tau^{a_{\tau, i}}$ where all $a_{\tau, i}\in \Z$ and $x_\tau$ denotes the algebraic character $K^\times \map L^\times$ given by the embedding $\tau$. If $\delta'_i=\delta_i$ for all $i\in \{1, \ldots, n\}$, then the point $z\in U_{\tri}^{\mathrm{reg}}(L)= X_{\tri, \mathrm{sp}}^{\mathrm{reg}, 1}(L)$. If  there exists a $\tau$ and  $i$, such that  and $a_{\tau, i}\neq 0$ for some $\tau$. By looking at the graded pieces of $\Fil_\bullet$ on $D$, we see that the $\tau$-Hodge-Tate-Sen weights of $r$ is given by 
\[\{\wt_{\tau}(\delta_1'), \ldots, \wt_{\tau}(\delta_n')\}=\{\wt_\tau(\delta_1)+a_{\tau, 1}, \ldots, \wt_{\tau}(\delta_n)+a_{\tau, n}\}
\]
On the other hand, by the condition (2) and (4) of Definition \ref{moduliprob} and by by Lemma \ref{htwt}, we know that the $\tau$-Hodge-Tate-Sen weights of $r$ is also
\[
\{\wt_{\tau}(\delta_1), \ldots, \wt_{\tau}(\delta_n)\}
\]
Thus for the $i$ where $a_{\tau, i}\neq 0$, we see that $\wt_\tau(\delta_i)+a_{\tau, i}=\wt_\tau(\delta_j)$ for some $j\neq i$, i.e. $\wt_{\tau}(\delta_i')=\wt_\tau(\delta_j')-a_{\tau, j}$. Now this immediately implies that there exists a nontrivial $w=(w_\tau)_{\tau\in \Sigma}\in \mathcal{S}$, such that 
\[
(\wt_{\tau}(\delta_{w_{\tau}(1)}'), \ldots, \wt_{\tau}(\delta_{w_{\tau}(n)}'))=(\wt_{\tau}(\delta_1), \ldots, \wt_{\tau}(\delta_n))
\]
as ordered tuples, and $\wt_{\tau}(\delta_{w_{\tau}(i)}')-\wt_{\tau}(\delta_i')\in \Z$ for any $\tau$ and $i\in \{1, \ldots, n\}$. This gives that the point $z'=(r, \delta_1', \ldots, \delta_n')\in X_{\tri, \mathrm{sp}}^{\mathrm{reg}, w}(L)$ and $z=i_w(z')$.

\end{proof}

\begin{lemma}\label{nocover}
Let $X$ be a rigid analytic space of dimension $n$ over $L$. Then $X$ cannot be covered by a countable union of Zariski-closed subanalytic spaces $\cup_{j\in \N}Y_j$ of dimension smaller than $n$.
\end{lemma}

\begin{proof}
We  may reduce the proof to the case $X$ is a closed unit polydisc of dimension $n$, given by the $L$-Banach algebra $L\langle T_1, \ldots, T_n\rangle$.  We may further assume without loss of generality that each $Y_j$ is given by the vanishing locus of a single nonzero element $f_j\in L\langle T_1, \ldots, T_n\rangle$. We prove by induction that there exists for each $j\in \N$, a finite extension $L_j$ of $L$, and a point $x_j=(a_{j, 1}, \ldots, a_{j, n})\in X(L_j)$, such that $f_l$ is a unit over the $L_j$-Banach algebra $L_j\langle\frac{T_1-a_{j, 1}}{p^j}, \ldots, \frac{T_n-a_{j,n}}{p^j}\rangle$ for any $l\leq j$ and that $x_{j+1}$ is a point of $\maxi(L_j\langle\frac{T_1-a_{j, 1}}{p^j}, \ldots, \frac{T_n-a_{j,n}}{p^j}\rangle)$. In other words, we have a sequence of shrinking rigid analytic closed polydiscs of radius $p^{-j}$ around $x_j$ where $f_j$ has no zeroes over. Assuming the result for $j$, and rename $S_{j, k}:=\frac{T_k-a_{j,k}}{p^j}$ for any $k\in \{1, \ldots, n\}$. By restriction now $f_j$ gives an element in $L_j\langle S_{j, 1}, \ldots, S_{j, n}  \rangle$. By scaling we may assume $f_j\in \Of_{L_j}\langle S_{j, 1}, \ldots, S_{j, n}  \rangle$ and $f_j\notin \pi_{L_j}\Of_{L_j}\langle S_{j, 1}, \ldots, S_{j, n}  \rangle$ where  $\Of_{L_j}$ has a uniformizer $\pi_{L_j}$ and residue field $k_{L_j}$. Then by reduction $\overline{f_j}$ gives a nonzero element in $k_{L_j}[S_{j, 1}, \ldots, S_{j, n}]$ and it is clear there exist an extension $k'/k_{L_j}$ and a $k'$-point given by $(S_{j,1}-\overline{b}_1, \ldots, S_{j,n}-\overline{b}_n)$ where $\overline{f_j}$ evaluates to a nonzero element in $k'$. Now choose any extension $L_{j+1}$ whose residue field contains $k'$ and a point $(b_1, \ldots, b_n)$, given in the coordinates $S_{j, 1}, \ldots, S_{j, n}$, that lifts $(\overline{b}_1, \ldots, \overline{b}_n)$. We see that in the coordinates $S_{j, 1}-b_1, \ldots, S_{j, n}-b_n$, $f_j$ evaluates to a power series whose constant coefficient is a unit and all other coefficients $\in \Of_{L_{j+1}}$, thus it becomes a unit in $\Of_{L_{j+1}}\langle \frac{S_{j, 1}-b_1}{p}, \ldots, \frac{S_{j, n}-b_n}{p} \rangle$. Setting $a_{j+1, k}:=a_{j, k}+p^jb_j$ for any $k\in \{1, \ldots, n\}$ finishes the induction step.

Now there exists a point $x$ in the intersection
\[
\bigcap_{j\in \N} \left(x_j+(p^j\Of_{C})^n \right) \subset \Of_C^n
\]
and it is immediately from our properties of $x_j$ that all $f_j$ is not zero at $x$.
\end{proof}

Now we can prove the main theorem. We first recall the setting and fix some notation from the constructions scattered in Section \ref{formdefprob} and \ref{laproof}:

Given a point $z=(r, \delta_1, \ldots, \delta_n)\in \spf (R^{\square}_{\overline{r}})^{\ad}_\eta(L) \times \mathcal{T}^n_1(L)$ such that $(\delta_1, \ldots, \delta_n)$ is regular. Decompose the $\tau$-Hodge-Tate-Sen weights of $\delta_1, \ldots, \delta_n$ into $k_\tau$ equivalence classes under the integral difference equivalence relations as in Definition \ref{intwteq}. We denote each equivalence classes by $S_{\tau, l}=\{a_{\tau,l, 1}, \ldots, a_{\tau, l, n_{\tau, l}}\}$ so that there exists an inverse bijection denoted by $a^{-1}: S_{\tau, l}\map \{1, \ldots, n_{\tau, l}\}$. If $z\in M_c(L)$ or $z\in X_\tri(L)$, the set of $\tau$-Hodge-Tate-Sen weights of $r$ is the same as $\{\wt_\tau(\delta_1), \ldots, \wt_\tau(\delta_n)\}$.  We have an associated permutation $w\in \mathcal{S}$ to $z$ as in Definition \ref{xtriassoc} and Definition \ref{assocwt}. We ordered the $\tau$-Hodge-Tate-Sen weights of $r$ that is of integral difference  with $\wt_\tau(\delta_i)$ for some $i\in S_{\tau, l}$, as  $h_{\tau,  l, 1}>\cdots > h_{\tau, l, n_{\tau, l}}$ for any fixed $l\in \{1, \ldots, k_\tau\}$. If $z\in M_c(L)$ or $r$ is trianguline with a triangulation $\Fil_\bullet$ on $D_\rig(r)$, the definition of $M_c$ or $\Fil_\bullet\invertt$ gives a triangulation on $D_{\rig}(r)\invertt$, and Definition \ref{assoc} gives a point $x\in X(L)$ associated to $r$ and the triangulatioin. We will freely state $x$ being associated with $z$ or $r$ (when there is a triangulation on $D_\rig(r)\invertt$) in the theorem below.

\begin{theorem}\label{maintheorem}
Given a point $z=(r, \delta_1, \ldots, \delta_n)\in \spf (R^{\square}_{\overline{r}})^{\ad}_\eta(L) \times \mathcal{T}^n_{1}(L)$ with $r$  regular. The following conditions are equivalent:
\begin{enumerate}
    \item $z\in X_{\tri}(\overline{r})(L)$.
    \item $z\in M_c(L)$ for some $c\geq 0$, and the asscoiated permutation $w\in \mathcal{S}$ satisfy $w\in \mathcal{S}(x)$, where $x$ is the associated point of $z$ in $X(L)$.
    \item $r$ is trianguline, having a triangulation with parameters $\delta_1', \ldots, \delta_n'$, such that  there exists a permutation $w\in \mathcal{S}(x)$, where $x$ is the associated point in $X(L)$ of $r$, such that 
    \[
    \delta_i=\delta_i'\prod_{\tau\in \Sigma}x_\tau^{h_{\tau, l, a^{-1}(w_{\tau}(i))}-\wt_\tau(\delta_i')}
    \]
    for any $i\in \{1, \ldots, n\}$.  Here $l$ is the unique one such that $i\in S_{\tau, l}$.
\end{enumerate}
\end{theorem}

\begin{proof}[Proof of Theorem \ref{maintheorem}]
We prove $(1)\Rightarrow (2) \Rightarrow (3)\Rightarrow (1)$. The hard part is $(3)\Rightarrow (1)$.

$(1)\Rightarrow (2)$: By \cite[3.7.1]{BHS}, or rather the proof of \cite[Theorem 6.3.13]{KPX}, we have that  the  $\phigamma$-module $\mathcal{M}:=D_\rig(r)\invertt$ over $\rob_{L, K}\invertt$ has a uniques triangulation with parameter $\delta_1, \ldots, \delta_n$. By Lemma \ref{uniquetri}, we see that for any $i\in \{1, \ldots, n\}$,
\[
\homo_{\varphi, \gamma_K}\left(\bigwedge^{n-i+1}D_\rig(r)\invertt, \rob_{L, K} (\prod_{j=i}^n \delta_j)\invertt\right)
\]
is a $1$-dimensional space over $L$. The above Hom space is also the increasing union of 
\[
\homo_{\varphi, \gamma_K}\left(\bigwedge^{n-i+1}D_\rig(r), t^{-c}\rob_{L, K} (\prod_{j=i}^n \delta_j)\right)
\]
over $c\map +\infty$. Thus we may find a $c$ sufficiently large such that 
\[
\homo_{\varphi, \gamma_K}\left(\bigwedge^{n-i+1}D_\rig(r), t^{-c}\rob_{L, K} (\prod_{j=i}^n \delta_j)\right)
\]
is $1$-dimensional over $L$ for any $i\in \{1, \ldots, n\}$. It clearly induces the unique triangulation on $\mathcal{M}$ with the given parameter. So condition (1) and (3) of Definition \ref{moduliprob} are satisfied for the point $z$. It follows from \cite[Proposition 2.9]{modular} (or \cite[6.2.12]{KPX}) that the $\tau$-Hodge-Tate-Sen weights of $r$ is the same as $\{\wt_\tau(\delta_1), \ldots, \wt_\tau(\delta_n)\}$. Thus, (2) and (4) of Definition \ref{moduliprob} is satisfied for the point $z$ by Lemma \ref{htwt}. We conclude that $z\in M_c(L)$.

Proposition \ref{nece} shows that $w\in \mathcal{S}(x)$ if $x\in X_\tri(L)$.

$(2)\Rightarrow (3)$: Since $z\in M_c(L)$, we have by definition a triangulation $\mathcal{M}_\bullet$ on $\mathcal{M}$ with parameter $\delta_1, \ldots, \delta_n$. This induces a triangulation $\Fil_\bullet:=\mathcal{M}_\bullet \cap D_{\rig}(r) $ of $D_\rig(r)$. Let $\delta_1', \ldots, \delta_n'$ be the parameters of this triangulation $\Fil$. Then as $\gr^i_{\Fil}\invertt=\gr^i_{\mathcal{M}_\bullet}$, we see that $\delta_i^{-1}\delta_i'$ is an algebraic character of $K^\times$ for all $i\in \{1, \ldots, n\}$. i.e.
\[
\delta_i=\delta_i'\prod_{\tau\in \Sigma}x_\tau^{d_{\tau, i}}
\]
for some integers $d_{\tau, i}$ and we solve these. In fact, by the definition of associated Weyl group element $w$ in the paragraphs preceding Definition \ref{assocwt}, we see that $\wt_\tau(\delta_{w_{\tau, l}^{-1}(a_s)})=h_{\tau, l, s}$ for any $\tau\in \Sigma$, $l\in \{1, \ldots, k_\tau\}$ and $s\in \{1, \ldots, n_{\tau, l}\}$, where $\{a_s\}$ is a listing of elements in $S_{\tau, l}$. In other words, renaming $i=w_{\tau, l}^{-1}(a_s)$, we see that $\wt_\tau(\delta_i)=h_{\tau, l, a^{-1}(w_{\tau, l}(i))}$. So $h_{\tau, l, a^{-1}(w_{\tau, l}(i))}=\wt_{\tau}(\delta_i)=\wt_\tau(\delta_i')+d_{\tau, i}$. Thus follows the formula for $d_{\tau, i}$.

$(3)\Rightarrow (1)$: Choose $c>n(h_{\tau, l, a^{-1}(w_{\tau}(i))}-\wt_\tau(\delta_i'))$ for any $\tau$ and $i$, we claim $z\in M_c(L)$: Using the given triangulation on $D_\rig(r)$,  we see as in the proof of Proposition \ref{descpmc}  that $\homo_{\varphi, \Gamma_K}\left(\bigwedge^{n-i+1}D_\rig(r), t^{-c}\rob_{Y} (\prod_{j=i}^n \delta_j)\right)$  must be of dimension $1$ over $L$ by Lemma \ref{uniquetri}. It is clear that those morphism spaces give the triangulation on $D_{\rig}(r)\invertt$ induced by the given one on  $D_{\rig}(r)$. Thus (1) and (3) of Definition \ref{moduliprob} are satisfied. Also the Sen operator acts on $D_{\mathrm{Sen}, \tau}(r)$ semisimply whose eigenvalues are $\bigcup_{l=1}^{k_\tau}\{h_{\tau, l, 1}, \ldots, h_{\tau, l, n_{\tau, l}}\}$, which is precisely the set $\{\wt_{\tau}(\delta_1), \ldots, \wt_{\tau}(\delta_n)\}$ by the definition of the characters $\delta_i$. (2) and (4) of Definition \ref{moduliprob} are satisfied  by Lemma \ref{htwt}.

Fix an affinoid neighborhood $U$ of $z$. Since $(\delta_1, \ldots, \delta_n)\in \mathcal{T}_1^n(L)$, we may assume $M_c\cap U \subset M_c^\regu$. We consider the set of irreducible components of $M_c\cap U$ passing through $z$. Combining Proposition  \ref{descpmc} and Proposition \ref{existmaxdim}, we see that $\widehat{M}_{c,z}$ has an irreducible component of maximal dimension  $[K:\Q_p]\frac{n(n+1)}{2}+n^2$ since $w\in \mathcal{S}(x)$.  We deduce that there exists an irreducible component $Y$ of $M_c\cap U$ passing through $z$ of dimension $[K:\Q_p]\frac{n(n+1)}{2}+n^2$. We claim that $Y\cap U^\regu_\tri$ is Zariski-dense in $Y$: Otherwise, $Y\cap U^\regu_\tri$ is contained in a Zariski-closed subanalytic space $Y_1$ of $Y$, necessarily of smaller dimension. Proposition \ref{covering} tells us that $Y$  is covered by the union of  $Y_1\supset Y\cap U^\regu_\tri=Y \cap X_{\tri, \mathrm{sp}}^{\mathrm{reg}, 1}$ and all $Y\cap i_w(X_{\tri, \mathrm{sp}}^{\mathrm{reg}, w})$ with $w$ ranging through all nontrivial elements of $\mathcal{S}$. Thus, by Lemma \ref{closedembed} and Lemma \ref{smallerd}, we see that $Y$ is covered by a countable union $\bigcup_{j\in \N}Y_j$ of Zariski-closed subanalytic spaces of smaller dimension than $[K:\Q_p]\frac{n(n+1)}{2}+n^2$. By Lemma \ref{nocover}, we arrive at a contradiction. Thus $Y\cap U^\regu_\tri$ is Zariski-dense in $Y$ and so $Y\subset X_\tri(\overline{r})$, the closure of $U^\regu_\tri$. In particular, $z\in X_\tri(\overline{r})(L)$.

\end{proof}

\bibliographystyle{amsalpha}
\bibliography{reference}

\end{document}